\documentclass[a4paper,11pt]{article}
\usepackage{amsmath,amsthm}
\usepackage{authblk}
\usepackage[style=numeric-comp,maxbibnames=99,bibencoding=utf8,firstinits=true]{biblatex}
\usepackage{booktabs}
\usepackage[makeroom]{cancel}
\usepackage{enumitem}
\usepackage[hmargin={26mm,26mm},vmargin={30mm,35mm}]{geometry}
\usepackage{loglogslopetriangle}
\usepackage{multirow}
\usepackage{nicefrac}
\usepackage[colorlinks,allcolors={blue}]{hyperref}
\usepackage{pgfplots}
\usepackage[compact,tiny]{titlesec}
\usepackage{tikz-cd}
\usepackage{subcaption}
\usepackage{tensor}

\usepackage{pgfplots,pgfplotstable}
\usepackage[font=small]{subcaption}

\usepackage{newtxtext}
\usepackage{newtxmath}
\usetikzlibrary{patterns}

\bibliography{main}

\newcommand{\email}[1]{\href{mailto:#1}{#1}}

\numberwithin{equation}{section}

\newtheorem{theorem}{Theorem}
\newtheorem{proposition}[theorem]{Proposition}
\newtheorem{lemma}[theorem]{Lemma}

\theoremstyle{remark}
\newtheorem{remark}[theorem]{Remark}
\theoremstyle{definition}

\newcommand{\st}{\,:\,}
\newcommand{\Real}{\mathbb{R}}

\DeclareRobustCommand{\bvec}[1]{\boldsymbol{#1}}
\newcommand{\uvec}[1]{\underline{\bvec{#1}}}
\newcommand{\cvec}[1]{\bvec{\mathcal{#1}}}

\DeclareMathOperator{\GRAD}{\bf grad}
\DeclareMathOperator{\CURL}{\bf curl}
\DeclareMathOperator{\DIV}{div}
\DeclareMathOperator{\ROT}{rot}
\DeclareMathOperator{\VROT}{\bf rot}

\newcommand{\compl}{{\rm c}}
 
\newcommand{\Hcurl}[1]{\bvec{H}(\CURL;#1)}

\newcommand{\Hdiv}[1]{\bvec{H}(\DIV;#1)}

\newcommand{\Xgrad}[2][]{\underline{X}_{\GRAD,#2}^{#1}}

\newcommand{\Xcurl}[2][]{\uvec{X}_{\CURL,#2}^{#1}}

\newcommand{\Xdiv}[2][]{\uvec{X}_{\DIV,#2}^{#1}}

\newcommand{\Xbullet}[2][]{\underline{X}_{\bullet,#2}^{#1}}

\newcommand{\sfP}{{\mathsf{P}}}
\newcommand{\sfQ}{{\mathsf{Q}}}

\newcommand{\Igrad}[2][]{\underline{I}_{\GRAD,#2}^{#1}}
\newcommand{\Icurl}[2][]{\uvec{I}_{\CURL,#2}^{#1}}
\newcommand{\Idiv}[2][]{\uvec{I}_{\DIV,#2}^{#1}}

\newcommand{\lproj}[2]{\pi_{\Poly{},#2}^{#1}}

\newcommand{\Gproj}[2]{\bvec{\pi}_{\cvec{G},#2}^{#1}}
\newcommand{\Gcproj}[2]{\bvec{\pi}_{\cvec{G},#2}^{\compl,#1}}

\newcommand{\Xproj}[2]{\bvec{\pi}_{\cvec{X},#2}^{#1}}
\newcommand{\Xcproj}[2]{\bvec{\pi}_{\cvec{X},#2}^{\compl,#1}}

\newcommand{\uGh}[1][]{\uvec{G}_h^{#1}}

\newcommand{\uCh}[1][]{\uvec{C}_h^{#1}}

\newcommand{\Dh}[1][]{D_h^{#1}}

\newcommand{\cGF}[1][]{\boldsymbol{\mathsf{G}}_F^{#1}}
\newcommand{\cGT}[1][]{\boldsymbol{\mathsf{G}}_T^{#1}}

\newcommand{\CF}[1][]{C_F^{#1}}
\newcommand{\cCT}[1][]{\boldsymbol{\mathsf{C}}_T^{#1}}

\newcommand{\DT}[1][]{D_T^{#1}}

\newcommand{\trF}[1][]{\gamma_F^{#1}}
\newcommand{\trFt}[1][]{\bvec{\gamma}_{{\rm t},F}^{#1}}

\newcommand{\faces}[1]{\mathcal{F}_{#1}}
\newcommand{\edges}[1]{\mathcal{E}_{#1}}

\newcommand{\FT}{\faces{T}}
\newcommand{\ET}[1][T]{\edges{#1}}
\newcommand{\EF}{\edges{F}}

\newcommand{\normal}{\bvec{n}}
\newcommand{\tangent}{\bvec{t}}

\newcommand{\Poly}[2][]{\mathcal{P}_{#1}^{#2}}
\newcommand{\vPoly}[2][]{\cvec{P}_{#1}^{#2}}
\newcommand{\Roly}[1]{\cvec{R}^{#1}}
\newcommand{\Goly}[1]{\cvec{G}^{#1}}
\newcommand{\cRoly}[1]{\cvec{R}^{\compl,#1}}
\newcommand{\cGoly}[1]{\cvec{G}^{\compl,#1}}

\newcommand{\norm}[2][]{\|#2\|_{#1}}

\newcommand{\vvvert}{\vert\kern-0.25ex\vert\kern-0.25ex\vert}

\newcommand{\Mh}[1][h]{\mathcal{M}_{#1}}
\newcommand{\Th}[1][h]{\mathcal{T}_{#1}}
\newcommand{\Fh}[1][h]{\mathcal{F}_{#1}}
\newcommand{\Eh}[1][h]{\mathcal{E}_{#1}}
\newcommand{\Vh}{\mathcal{V}_h}

\newcommand{\Pgrad}[1][]{P_{\GRAD,T}^{#1}}
\newcommand{\Pcurl}[1][]{\bvec{P}_{\CURL,T}^{#1}}
\newcommand{\Pgradh}[1][]{P_{\GRAD,h}^{#1}}
\newcommand{\Pcurlh}[1][]{\bvec{P}_{\CURL,h}^{#1}}
\newcommand{\Pdiv}[1][]{\bvec{P}_{\DIV,T}^{#1}}
\newcommand{\Pbullet}[1][]{P_{\bullet,T}^{#1}}
\newcommand{\Pbulleth}[1][]{P_{\bullet,h}^{#1}}

\newcommand{\pf}[2][p]{\Omega^{#1}(#2)} 
\newcommand{\Lapf}[2][p]{\pf[#1]{#2}\otimes\La} 
\newcommand{\Gcd}[1]{D_{#1}} 
\newcommand{\Ed}{d} 
\newcommand{\Hstar}{{\star}} 

\newcommand{\La}{\mathfrak{g}}
\newcommand{\LaXgrad}[2][]{\Xgrad[#1]{#2}\otimes\La}
\newcommand{\LaXcurl}[2][]{\Xcurl[#1]{#2}\otimes\La}
\newcommand{\LaXdiv}[2][]{\Xdiv[#1]{#2}\otimes\La}
\newcommand{\LaXbullet}[2][]{\Xbullet[#1]{#2}\otimes\La}
\newcommand{\LaHcurl}[1]{\bvec{H}(\CURL;#1)\otimes\La}
\newcommand{\LaHdiv}[1]{\bvec{H}(\DIV;#1)\otimes\La}
\newcommand{\LaHgrad}[1]{H^1(#1)\otimes\La}
\newcommand{\LaIgrad}[2][\La]{\Igrad[#1]{#2}}
\newcommand{\LaIcurl}[2][\La]{\Icurl[#1]{#2}}

\newcommand{\LauGh}[1][\La]{\uGh[#1]}
\newcommand{\LauCh}[1][\La]{\uCh[#1]}

\newcommand{\LaDh}[1][\La]{\Dh[#1]}
\newcommand{\LatrFt}[1][\La]{\trFt[#1]}
\newcommand{\LaCF}[1][\La]{\CF[#1]}
\newcommand{\LaPgrad}[1][\La]{\Pgrad[#1]}
\newcommand{\LaPcurl}[1][\La]{\Pcurl[#1]}
\newcommand{\LaPgradh}[1][\La]{\Pgradh[#1]}
\newcommand{\LaPcurlh}[1][\La]{\Pcurlh[#1]}

\newcommand{\LaPbullet}[1][\La]{\Pbullet[#1]}
\newcommand{\LaPbulleth}[1][\La]{\Pbulleth[#1]}
\newcommand{\ebkt}[3][]{\Hstar[#2,#3]^{#1}}
\newcommand{\ebkttr}[3][\DIV,h]{\ebkt[#1]{#2}{#3}}
\newcommand{\ebkttrk}[3][\DIV,k,h]{\ebkt[#1]{#2}{#3}}
\newcommand{\dbkt}[2]{\ebkt{#1}{\Hstar#2}}
\newcommand{\ymE}{\bvec{E}}
\newcommand{\ymB}{\bvec{B}}
\newcommand{\ymA}{\bvec{A}}

\newcommand{\deltat}{\delta\hspace*{-0.15ex}t}

\pgfplotsset{select coords between index/.style 2 args={
    x filter/.code={
        \ifnum\coordindex<#1\fi
        \ifnum\coordindex>#2\fi
    }
}}

\begin{document}

\title{A polyhedral discrete de Rham numerical scheme for the Yang--Mills equations}

\author[1]{J\'er\^ome Droniou}
\author[1]{Todd A. Oliynyk}
\author[1]{Jia Jia Qian}

\affil[1]{School of Mathematics, Monash University, Melbourne, Australia, \email{jerome.droniou@monash.edu}, \email{todd.oliynyk@monash.edu}, \email{jia.qian@monash.edu}}

\maketitle

\begin{abstract}
  ~
  We present a discretisation of the 3+1 formulation of the Yang--Mills equations in the temporal gauge, using a Lie algebra-valued extension of the discrete de Rham (DDR) sequence, that preserves the non-linear constraint exactly. In contrast to Maxwell's equations, where the preservation of the analogous constraint only depends on reproducing some complex properties of the continuous de Rham sequence, the preservation of the non-linear constraint relies for the Yang--Mills equations on a constrained formulation, previously proposed in \cite{Christiansen.Winther:06}. The fully discrete nature of the DDR method requires to devise appropriate constructions of the non-linear terms, adapted to the discrete spaces and to the need for replicating the crucial Ad-invariance property of the $L^2$-product. We then prove some energy estimates, and provide results of 3D numerical simulations based on this scheme. 
\end{abstract}


\section{Introduction}\label{sec:intro}

In this work we propose a numerical method to approximate the Yang--Mills equations on a generic 3D polyhedra mesh. The scheme is based on the recent discrete de Rham polytopal complex, which has an arbitrary degree of accuracy and reproduces the calculus properties (and cohomology properties) of the continuous de Rham complex. Using an Lagrange-multiplier augmented formulation, the scheme satisfies a discrete version of the non-linear constraint condition associated to the Yang--Mills equations. We establish uniform $L^2$-bounds on the electric and magnetic field, showing the stability of the scheme on these variables, and present a series of numerical tests using the lowest-order version of the scheme on classical (tetrahedral, cubic) and polyhedral (Voronoi) meshes, showing the expected rates of convergence and the practical preservation of the discrete constraint.

Quantization of the Yang--Mills equations yields quantum field theories that govern the interactions of elementary particles, which, for a particular choice of the gauge group, form the basis of the Standard Model of particle physics. As a classical field theory, the Yang--Mills equations can be viewed as natural, non-linear generalization of Maxwell's equations, which coincide with Maxwell's equations for the gauge group $U(1)$. As with Maxwell's equations, solutions to the Yang-Mills equations are obtained by solving an initial value problem where the initial data cannot be chosen arbitrarily, but must satisfy certain \textit{constraint equations}. At the continuous level, these constraints, if satisfied initially, can be shown, as a consequence of the evolution equations, to \textit{propagate}, that is, they will continue to be satisfied for all later times. However, at the discrete level, constraint propagation is difficult to achieve exactly and the lack of perfect constraint propagation is a known source of instabilities that lead to exponential growth in errors in numerical simulations of the Yang-Mills. Similar
issues involving constraint propagation also occur in numerical simulations of Einstein's equations \cite{alic_et_al:2012,Brodbeck_et_al:1999,FrauendienerVogel:2005}. We view the Yang-Mills equations as an important model for testing new numerical methods that could eventually be applied to Einstein's equations.   

Constraint propagation in numerical schemes for the Maxwell's equations only relies on satisfying, at the discrete level, the calculus relation $\DIV\CURL=0$. Due to non-linearity that is present in the Yang-Mills equations for non-Abelian choices of the gauge group, such a discrete property is not sufficient to enforce exact constraint propagation for the scheme. There are two main approaches for addressing constraint propagation at the discrete level in the Yang-Mills equations, and more generally in non-linear evolution equations involving constraints such as Einstein's equations. The first approach is called \textit{constraint damping}, which has been influential in controlling the constraint violations in numerical simulations of Einstein's equations \cite{alic_et_al:2012,Brodbeck_et_al:1999,Gundlach_et_al:2005}; see also \cite{Calabrese:2004} for use of this approach in numerically solving Maxwell's equations. In this approach, 
suitable combinations of the constraint equations are added to the evolution equations in order to ensure that constraint violations
remain bounded during numerical simulations. When successful, constraint damping results in numerical solutions with acceptable amounts of constraint violating errors.

The second main approach to constraint propagation is to solve or otherwise enforce the constraint equations exactly at the discrete level. This is the approach we take here using a constrained formulation of the Yang-Mills equations that is inspired by the constrained version of the Yang-Mills equations from \cite{Christiansen.Winther:06}. Although, in theory, any constraint can be added to a system through the introduction of Lagrange multipliers, it must be noted that, when done in an improper way, such an addition can lead to destruction of the system's properties, such as energy conservation/decay, and lead to numerical instabilities, see \cite[Section IV.4]{Hairer.Lubich.ea:06} and \cite{Detweiler:1987} for a discussion of this phenomena in the context of Einstein's equations. However, the approach of \cite{Christiansen.Winther:06} that we use here does not present this issue, and leads to a constraint-preserving scheme that is also numerically stable (as can be seen through energy estimates on the potential and electric field).

The discrete de Rham (DDR) complex, developed in \cite{Di-Pietro.Droniou.ea:20,Di-Pietro.Droniou:21*2}, is a fully discrete version of the de Rham complex, which reproduces at the discrete level the calculus properties (e.g. $\CURL\GRAD=0$, $\DIV\CURL=0$, etc.), as well as the cohomological properties of the continuous complex \cite{Di-Pietro.Droniou.ea:22}, including its exactness properties in case of trivial domain topology. As demonstrated, e.g., in \cite{Di-Pietro.Droniou:21*1}, these properties are essential in some models to design robust schemes. The two main features of DDR is its arbitrary degree of accuracy, and its capability to handle generic polytopal meshes -- that is, meshes made of generic polytopal elements (as opposed to finite element methods (FEM) which are mainly restricted to tetrahedral or hexahedral meshes). These flexibilities enable the resulting schemes to better capture steep behaviour of solutions at a reduced computational cost \cite{Di-Pietro.Droniou:21}; additionally, the higher-order version of DDR is amenable to generic cost-reducing techniques such as static condensation and -- probably more interesting for non-linear models such as the Yang--Mills equations -- serendipity processes \cite{Di-Pietro.Droniou:22}. We note that the DDR approach can be bridged with the virtual element method (VEM), see \cite{Beirao-da-Veiga.Brezzi.ea:18*2}, as recently shown in \cite{Beirao-da-Veiga.Dassi.ea:22}.

The first contribution of this work is the design and stability analysis of a DDR scheme for the Yang--Mills equations. The usage of polytopal meshes for these equations provide increased flexibility to handle steep coefficients via local mesh refinement; such coefficients could appear in the presence of a background metric with high curvature (even though we consider, for ease of presentation, the case of a flat space in this work). The scheme we present can not only handle polytopal meshes, but can also be written at an arbitrary degree of accuracy, which enhances its capacity to capture fine behaviours of the solution.

The principle to obtain a scheme from the DDR complex is to replace, in the continuous weak formulation (based, here, on the potential and electric fields), the spaces, $L^2$-inner products and differential operators by the discrete DDR spaces, inner products and operators.
Due to the complex property of the DDR sequence, in the linear case (Maxwell equations), this leads to a scheme that automatically preserves the zero-divergence constraint. However, the general Yang--Mills equations have a non-linear constraint, whose preservation requires, besides the discrete version of the $\DIV\CURL=0$ relation, the introduction in the scheme of an additional equation and its related Lagrange multiplier; this was already noticed, e.g., in \cite{Christiansen.Winther:06} for methods using conforming complexes such as the classical low-order FEM complex on triangles.

As the DDR method is based on a fully discrete approach (discrete spaces that are not sub-spaces of the corresponding continuous ones), a challenge to overcome in the design of the scheme lies in the definition of adequate Lie brackets acting on the discrete DDR spaces that preserves the required Ad-invariance properties and enables the definition of a suitable magnetic field; such a challenge does not exist when considering conforming methods with explicit representation of the basis functions, such as the FEM in \cite{Christiansen.Winther:06} -- but these methods are restricted to specific underlying meshes, with far less flexibility than polytopal meshes.

Another noticeable contribution to our work is the usage of a fully non-linear (not staggered) scheme in time; such an approach allows for easy higher-accuracy in time (e.g., through Crank--Nicolson or Runge--Kutta discretisations). We also numerically demonstrate the efficiency of the standard Newton algorithm for solving this non-linear scheme: 3D simulations on a variety of polytopal meshes show that about 2-3 iterations are required for Newton to produce a solution with a very small residual. In passing, it should be mentioned that 3D simulations for the Yang--Mills equations are not commonly considered in the literature on mesh-based numerical schemes, with simulations often only run in 2D models \cite{Christiansen.Winther:06,Berchenko-Kogan.Stern:21}.

The paper is organised as follows. In Section \ref{sec:ym.equations} we recall the Yang--Mills equations in flat 3+1 formalism using the temporal gauge. Section \ref{sec:DDR} gives a brief presentation of the lowest-order discrete de Rham complex (spaces, discrete operators, inner products), starting from its scalar-valued version before showing how simple tensorisation gives a Lie algebra-valued complex. In Section \ref{sec:schemes} we show how schemes can be built from this complex, starting from the simpler case of the Maxwell equations before dealing with the non-linear Yang--Mills equations. After having recalled how the constraint is preserved in the continuous equations, we design a scheme for these equations that exactly preserves the constraint at the discrete level, which requires us to build appropriate discretisations of two types of non-linear terms involving the Lie bracket. We also establish energy estimates on the scheme. Section \ref{sec:tests} presents the results of 3D numerical simulations on a variety of meshes (tetrahedral, cubic and Voronoi tesselations); these tests show the practical efficiency of the Newton algorithm to solve the non-linear scheme, the expected order 1 rate of convergence of the approximate solution, and the preservation of the constraint up to round-off errors.
In Section \ref{sec:DDR.high.order} we show how the low-order Lie algebra DDR complex and resulting schemes built can be easily extended to arbitrary-order complexes and schemes that satisfy the same preservation of constraint and energy estimates as the low-order version. Finally, a conclusion is given in Section \ref{sec:conclusion}, highlighting possible extensions to other polytopal methods such as VEM, and future research questions on the topic.

\section{The Yang--Mills equations in flat $3+1$ formalism}\label{sec:ym.equations}

Let $G$ be a compact Lie group with the associated (finite dimensional) Lie algebra $\La$: a vector space equipped with a Lie bracket, that is, an antisymmetric bilinear map $[\cdot,\cdot]:\La\times\La\to\La$ satisfying the Jacobi identity: 
\begin{equation*}
[a,[b,c]]+[b,[c,a]]+[c,[a,b]]=0\qquad\forall a,b,c\in\La.
\end{equation*}
Let $\langle\cdot,\cdot\rangle:\La\times\La\to\Real$ be an Ad-invariant inner product on $\La$, which implies:
\begin{equation}\label{eq:ad.inv}
\langle[a,b],c\rangle=\langle a,[b,c]\rangle\qquad\forall a,b,c\in\La.
\end{equation}

For $U\subset\Real^3$ with the standard Euclidean metric, we denote by $\Hstar$ the associated Hodge star operator (linear map from $p$-forms to $(3-p)$-forms), and $\Lapf{U}$ the space of Lie algebra-valued $p$-forms on $U$. Taking a basis $(e_I)_{I=1}^{\dim \mathfrak{g}}$ of the Lie algebra, we can decompose any $\phi\in\Lapf{U}$ as $\phi=\phi^I\otimes e_I$, where $\phi^I\in\pf{U}$ and, here and in the following, we adopt the Einstein summation convention (summing over indices repeated at the top and bottom). We can then define the exterior derivative $\Ed:\Lapf{U}\to\Lapf[p+1]{U}$ by $\Ed\phi=\Ed\phi^I\otimes e_I$, and the $L^2$-inner product on $\Lapf{U}$ as
\begin{equation*}
  \int_U\phi^I\wedge\Hstar\psi^J\langle e_I,e_J\rangle,\qquad\forall\phi,\psi\in\Lapf{U}.
\end{equation*}
The bracket of Lie algebra-valued $p$ and $q$ forms is the Lie algebra-valued $(p+q)$-form given by
\begin{equation*}
  [\phi\wedge\psi]\coloneq\phi^I\wedge\psi^J[e_I,e_J],\qquad\forall\phi\in\Lapf{U},\quad\forall\psi\in\Lapf[q]{U}.
\end{equation*}
These definitions are of course independent of the choice of basis on $\La$.

For a local gauge potential $A\in\Lapf[1]{U}$, the gauge covariant derivative is $\Gcd{A}:\Lapf{U}\to\Lapf[p+1]{U}$, defined by $\Gcd{A}\psi=\Ed\psi+[A\wedge\psi]$, and we denote by $\Gcd{A}^*$ its adjoint with respect to the $L^2$-inner product. A 3+1 first order formulation of the Yang--Mills equations, in the temporal gauge, is: Find $A,E\in C^1([0,T];\Lapf[1]{U})$ such that, setting $B=\Ed A + \frac12 [A\wedge A]$,
\begin{equation}\label{eq:ym.df}
    \partial_t A=-E\quad\mbox{ and }\quad\partial_t E=\Gcd{A}^* B.
\end{equation}
One can show that this system preserves the constraint $\Gcd{A}^* E=0$ at all times if it is satisfied initially (see Proposition \ref{prop:ym.preservation.constraint} below for a proof for the weak formulation).

To make it easier for our discretisation, which is based on a de Rham complex of scalar and vector valued functions with the operators $\GRAD, \CURL, \DIV$, we will reformulate the Yang-Mills equations in vector notations, which will then be used in the remainder of the paper.  

Denoting our standard Euclidean metric by $g$ ($g_{\mu\nu}=\delta_{\mu\nu}$), we use the musical isomorphism to identify vector fields and 1-forms: $\flat:\mathfrak{X}(U)\otimes\La\to\Lapf[1]{U}, X\mapsto X^\flat$, where the components of $X^\flat$ are written without the flat, $X_\mu\coloneq(X^\flat)_\mu=g_{\mu\nu}X^\nu=X^\mu$. The inverse operator $\sharp\coloneq\flat^{-1}$ sends a 1-form $\omega$ to the vector field $\omega_\sharp$, with again components written without the musical symbol, $\omega^\mu\coloneq(\omega_\sharp)^\mu=g^{\mu\nu}\omega_\nu=\omega_\mu$, and $g^{\mu\nu}=\delta^{\mu\nu}$ are components of the inverse metric $g^{-1}$, characterised by $g_{\mu\nu}g^{\nu\gamma}=\delta_\mu^\gamma$. 
We note that, given the choice of the standard Euclidean metric here, these identifications and operators are somewhat trivial, but we made them explicit for the sake of clarity (and for ease of generalisation in case a different metric is selected).

We then have the correspondence
\begin{alignat*}{4}
  (\Ed f)_\sharp ={}&\GRAD f, \qquad{} &\forall f&\in C^\infty(U)\otimes\La,\\
  (\Hstar\Ed\phi)_\sharp={}&\CURL\phi_\sharp, \qquad{}&\forall\phi &\in\Lapf[1]{U},\\
  \Hstar\Ed\psi={}&\DIV(\Hstar\psi)_\sharp,\qquad{}&\forall\psi&\in\Lapf[2]{U},
\end{alignat*}
where the last equation is also equivalent to
\begin{equation*}
  \Hstar\Ed\Hstar\phi=\DIV\phi_\sharp,\qquad\forall \phi\in\Lapf[1]{U}.
\end{equation*}
This implies that, under vanishing natural boundary conditions, the adjoint $\Ed^*$ of $d$ has relations 
\begin{alignat*}{4}
  \Ed^*\phi={}& -\DIV\phi_\sharp, \quad& \forall \phi\in\Lapf[1]{U},\\
  (\Ed^*\psi)_\sharp={}&\CURL(\Hstar\psi)_\sharp, \quad& \forall \psi\in\Lapf[2]{U},\\
  (\Hstar\Ed^*\varphi)_\sharp={}&-\GRAD\Hstar\varphi,\quad&\forall \varphi\in\Lapf[3]{U},
\end{alignat*}
Using the Ad-invariance property \eqref{eq:ad.inv}, the wedge bracket $[A\wedge\cdot]$ has the adjoint
\begin{alignat*}{2}
  [A\wedge\cdot]^*\phi={}&(-1)^{p+1}\Hstar[A\wedge \Hstar\phi], \qquad \forall \phi\in\Lapf{U}.
\end{alignat*} 
We will now indicate Lie algebra-valued vector fields using boldface, and for these vector fields $\bvec{v}$, the notation $\bvec{v}_\mu$ continues to refer to $(\bvec{v}^\flat)_\mu=g_{\mu\nu}\bvec{v}^\nu$, where the index is lowered through the metric $g$. The dot product of vectors can then be written as $\bvec{v}\cdot\bvec{w}=g_{\mu\nu}\bvec{v}^\mu\bvec{w}^\nu=\bvec{v}^\mu \bvec{w}_\mu$, whenever the multiplication is defined. We require vector versions of the wedge bracket, each corresponding to particular combinations of $p$, $\Hstar$ and $[\cdot\wedge\cdot]$. For a Lie algebra-valued vector field $\bvec{v}$, we define the maps
\begin{alignat*}{4}
  [\cdot,\cdot]:{}&(\mathfrak{X}(U)\otimes\La)\times(C^\infty(U)\otimes\La)\to\mathfrak{X}(U)\otimes\La, \quad &[\bvec{v}, f]\coloneq{}&[\bvec{v}^\flat\wedge f]_\sharp,\\
  \Hstar[\cdot,\cdot]:{}&(\mathfrak{X}(U)\otimes\La)\times(\mathfrak{X}(U)\otimes\La)\to\mathfrak{X}(U)\otimes\La, \quad&\Hstar[\bvec{v},\bvec{w}]\coloneq{}&\Hstar[\bvec{v}^\flat\wedge\bvec{w}^\flat]_\sharp,\\
  \Hstar[\cdot,\Hstar\cdot]:{}&(\mathfrak{X}(U)\otimes\La)\times(\mathfrak{X}(U)\otimes\La)\to C^\infty(U)\otimes\La, \quad&\Hstar[\bvec{v},\Hstar\bvec{w}]\coloneq{}&\Hstar[\bvec{v}^\flat\wedge\Hstar\bvec{w}^\flat].
\end{alignat*}
From these definitions, these maps are bilinear and the (Lie algebra-valued) components of the resulting vector fields are (with the Lie bracket of functions on the right)
\begin{alignat*}{4}
  [\bvec{v}, q]^\alpha={}&[\bvec{v}^\alpha, q],\quad
  \Hstar[\bvec{v},\bvec{w}]^\alpha={}&\varepsilon\indices{^\alpha_{\mu\nu}}[\bvec{v}^\mu,\bvec{w}^\nu],\quad
  \Hstar[\bvec{v},\Hstar\bvec{w}]={}&[\bvec{v}^\alpha,\bvec{w}_\alpha],
\end{alignat*}
where $\varepsilon_{\alpha\mu\nu}$ is the volume form; these equations clearly show that the second map is symmetric and the third is anti-symmetric. The $L^2$-inner products for scalar and vector Lie algebra-valued functions are, respectively,
\begin{equation*}
  \int_U \langle q,r\rangle\coloneq\int_U q^Ir^J\langle e_I,e_J\rangle, \qquad
  \int_U \langle \bvec{v},\bvec{w}\rangle\coloneq\int_U\langle\bvec{v}^\mu,\bvec{w}_\mu\rangle=\int_U(\bvec{v}^\mu)^I(\bvec{w}_\mu)^J\langle e_I, e_J \rangle.
\end{equation*}

Now we can recast the Yang--Mills equations \eqref{eq:ym.df} using the vector fields $\ymA\coloneq A_\sharp$, $\ymE\coloneq E_\sharp$, and $\ymB\coloneq(\Hstar B)_\sharp=\CURL\ymA+\frac12\ebkt{\ymA}{\ymA}$, and forget completely the underlying differential forms: Find $\ymA,\ymE\in C^1([0,T];\mathfrak{X}(U)\otimes\La)$ such that
\begin{subequations}\label{eq:ym}
\begin{alignat}{3}
    \partial_t\ymA&=-\ymE \quad&\mbox{in }U, \\ 
    \partial_t\ymE&=\CURL\ymB+\ebkt{\ymA}{\ymB}\quad&\mbox{in }U,
    \label{eq:ym.ev.E}
\end{alignat}
with the natural boundary conditions
\begin{equation}\label{eq:bcs}
  \ymB\times\normal=0\quad\mbox{ and }\quad\ymE\cdot\normal=0\quad\mbox{on }\partial U
\end{equation}
and initial conditions
\begin{equation}
  \ymA(0)=\ymA_0\quad\mbox{ and }\quad\ymE(0)=\ymE_0\quad\mbox{ in $U$}.
\end{equation}
\end{subequations}
The associated constraint $\Gcd{A}^* E=0$ is written
\begin{equation}\label{eq:ym.const.E}
  \DIV\ymE+\ebkt{\ymA}{\Hstar\ymE}=0\quad\mbox{in } U.
\end{equation}

\section{Lowest order Discrete de Rham complex}\label{sec:DDR}

We describe here the lowest-order version of the discrete de Rham complex, which is at the core of our polytopal method for the Yang--Mills equations.

\subsection{Notations}\label{sec:mesh}

Let $U\subset\Real^3$ denote a connected polyhedral domain. We use the mesh definitions and notations of \cite{Di-Pietro.Droniou:21*2}. The mesh is a partition of the domain into polyhedral elements gathered in the set $\Th$; the mesh faces are gathered in $\Fh$, the mesh edges in $\Eh$ and the mesh vertices in $\Vh$. For each $\sfP\in\Mh\coloneq\Th\cup\Fh\cup\Eh\cup \Vh$ and each $\mathcal X\in\{\mathcal T,\mathcal F,\mathcal E,\mathcal V\}$, we denote by $\mathcal X_\sfP$ the set of mesh entities $\sfQ$ identified by $\mathcal X$ such that $\sfQ\subset\sfP$ (if $\sfP$ has a higher dimension than $\sfQ$) or $\sfP\subset\sfQ$ (otherwise). $h_{\sfP}$ is the diameter of $\sfP$, and we set $h\coloneq\max_{T\in\Th}h_T$.

Faces and edges are oriented: we select for each $F\in\Fh$ a unit normal $\normal_F$ to $F$, and for each $E\in\Eh$ a unit tangent $\tangent_E$ to $E$. The relative orientation of a face $F\in\FT$ in an element $T$ is $\omega_{TF}=+1$ if $\normal_F$ points outside $T$, and $\omega_{TF}=-1$ otherwise. For each $F\in\Fh$ and $E\in\EF$, we denote by $\normal_{FE}$ the unit normal to $E$ in the plane spanned by $F$ such that $(\tangent_E,\normal_{FE},\normal_F)$ forms a right-handed system of coordinate; the orientation $\omega_{FE}\in\{-1,+1\}$ of $E$ respective to $F$ is set such that $\omega_{FE}\normal_{FE}$ points outside $F$ (in the plane spanned by $F$).

For each $F\in\Fh$, $\GRAD_F$ and $\DIV_F$ are the tangent gradient and divergence operators and, for $r:F\to\Real$ and $\bvec{z}:F\to\Real^2$ smooth enough, we let
$\VROT_F r\coloneq (\GRAD_F r)^\perp$ and
$\ROT_F\bvec{z}=\DIV_F(\bvec{z}^\perp)$,
with $\perp$ denoting the rotation of angle $-\frac\pi2$ in the tangent space to $F$ oriented counter-clockwise with respect to $\normal_F$.

For all $\sfP\in\Mh$ and natural number $k\ge 0$, $\Poly{k}(\sfP)$ is the space spanned by the restriction to $\sfP$ of polynomials in $\Real^3$ of degree $\le k$; this space is isomorphic to the space of $d$-valuate polynomials, where $d$ is the dimension of $\sfP$.
The $L^2$-orthogonal projector onto $\Poly{k}(\sfP)$ is $\lproj{k}{\sfP}:L^2(\sfP)\to\Poly{k}(\sfP)$. The spaces of vector-valued polynomials $\sfP\to\Real^d$ (with, as above, $d$ being the dimension of $\sfP$) and corresponding projectors are denoted in the same way but using boldface fonts.

\subsection{Scalar DDR complex}\label{sec:scalar.DDR}

The DDR complex corresponds to a discrete version of de Rham complex, in which the spaces $H^1(U)$, $\Hcurl{U}$, $\Hdiv{U}$ and $L^2(U)$ are replaced by discrete counterparts, linked together with discrete operators mimicking $\GRAD$, $\CURL$ and $\DIV$. We give here a brief overview of the lowest-order version of the DDR complex, as this is the one we will use to design a scheme for the Yang--Mills equations in Section \ref{sec:ym}; the extension to higher-order versions is covered in Section \ref{sec:DDR.high.order}, and we refer to \cite{Di-Pietro.Droniou.ea:20,Di-Pietro.Droniou:21*2} for a more detailed presentation. For the lowest-order version of the DDR, the design of spaces and discrete vector calculus operators are strongly related to that encountered in Compatible Discrete Operators (CDO) and Discrete Geometric Analysis theories (see \cite{Bonelle:14} and reference therein), but the approach to the construction of the inner products in the spaces differ: where, for example, CDO bases this design on the choice of appropriate discrete Hodge operators, DDR rather creates potential reconstructions and inner products based on their integrals and stabilisation terms (as in other arbitrary-order polytopal methods, such as Hybrid High-Order or Virtual Element Methods \cite{Di-Pietro.Droniou:20,Beirao-da-Veiga.Brezzi.ea:14}). An added advantage of this potential-based approach is that it readily provides the required tools to discretise the non-linear terms in the Yang--Mills equations.

The driving design behind the discrete spaces and operators of the DDR complex are integration-by-parts formula for the vector calculus operators. These formulas justify the choice of degrees of freedom in the spaces, the construction of the discrete differential operators, as well as the design of the potentials (reconstructions of piecewise polynomial functions on the faces or elements based on the degrees of freedom in each space).

\subsubsection{Spaces and discrete vector calculus operators}

The spaces and interpolators (which give meaning to the components of the vectors in the spaces) of the lowest-order DDR complex are as follows. The discrete counterpart of the $H^1(U)$ space is made of the continuous piecewise linear functions on the mesh skeleton:
\begin{equation*}
  \Xgrad{h}\coloneq\Big\{
  q_{\Eh}\in C^0(\Eh)\st
    \text{$q_E\coloneq (q_{\Eh})_{|E}\in\Poly{1}(E)$ for all $E\in\Eh$}
    \Big\}.
\end{equation*}
The interpolator $\Igrad{h}:C^0(\overline{U})\to\Xgrad{h}$ is defined such that, for all $v\in C(\overline{U})$, $\Igrad{h}v$ is the unique piecewise polynomial on $\Eh$ such that $(\Igrad{h}v)(\bvec{x}_V)=v(\bvec{x}_V)$ for all $V\in\Vh$, where $\bvec{x}_V$ is the coordinate of $V$. 

The discrete $\Hcurl{U}$ space is
\begin{equation*}
  \Xcurl{h}\coloneq\Big\{
    \uvec{v}_h
    =\big(
    (v_E)_{E\in\Eh}
    \big)\st\text{$v_E\in\Real$ for all $E\in\Eh$}\Big\}
\end{equation*}
and we interpolate $\bvec{v}\in \bvec{C}^0(\overline{U})\to \Xcurl{h}$ by setting $\Icurl{h}\bvec{v}=(\lproj{0}{E}(\bvec{v}\cdot\tangent_E))_{E\in\Eh}$; this shows that the components $(v_E)_{E\in\Eh}$ of a vector in $\Xcurl{h}$ play the role of tangential values along the edges.

The discrete $\Hdiv{U}$ space is
\begin{equation*}
  \Xdiv{h}\coloneq\Big\{
    \uvec{w}_h
    =\big((w_F)_{F\in\Fh}\big)\st
    \text{$w_F\in\Real$ for all $F\in\Fh$}
    \Big\}.
\end{equation*}
For $\bvec{w}\in H^1(U)$, we set $\Idiv{h}\bvec{w}=(\lproj{0}{F}(\bvec{w}\cdot\normal_F))_{F\in\Fh}$; the components $(w_F)_{F\in\Fh}$ in $\Xdiv{h}$ therefore play the role of normal fluxes across the faces.

\begin{remark}[Simpler presentation of the spaces]
Noticing that an element in $\Xgrad{h}$ is entirely determined by its values at the vertices, we have $\Xgrad{h}\approx \Real^{\Vh}$, $\Xcurl{h}\approx \Real^{\Eh}$ and $\Xdiv{h}\approx \Real^{\Fh}$. We chose the slightly longer presentations
above as they also allow us to fix the notations for the components of each vector onto specific mesh entities.
\end{remark}

The discrete de Rham sequence, which is a complex with the same cohomology as the continuous de Rham complex (it is in particular exact if $U$ is contractible) \cite{Di-Pietro.Droniou.ea:20,Di-Pietro.Droniou.ea:22}, then reads
\begin{equation*}
  \begin{tikzcd}
    \Real\arrow{r}{\Igrad{h}} & \Xgrad{h}\arrow{r}{\uGh} & \Xcurl{h}\arrow{r}{\uCh} & \Xdiv{h}\arrow{r}{\Dh} & \Poly{0}(\Th)\arrow{r}{0} & \{0\}
  \end{tikzcd}
\end{equation*}
where $\Poly{0}(\Th)$ is the space of piecewise constant functions on $\Th$, and the discrete gradient, curl and divergence are defined by
\begin{alignat}{4}
\uGh\underline{q}_h{}&=(q_E')_{E\in\Eh}&\quad\forall\underline{q}_h\in\Xgrad{h},\label{eq:def.uGh}\\
\uCh\uvec{v}_h{}&=(\CF\uvec{v}_F)_{F\in\Fh}\quad\mbox{ with }\quad\CF\uvec{v}_F=-\sum_{E\in\EF}\omega_{FE}|E|v_E\quad\forall F\in\Fh&\quad\forall\uvec{v}_h\in\Xcurl{h},\nonumber\\
\Dh\uvec{w}_h{}&=(\DT\uvec{w}_T)_{T\in\Th}\quad\mbox{ with }\quad\DT\uvec{w}_T=\frac{1}{|T|}\sum_{F\in\EF}\omega_{TF}|F|w_F\quad\forall T\in\Th&\quad\forall\uvec{w}_h\in\Xdiv{h}.\nonumber
\end{alignat}
In \eqref{eq:def.uGh}, for each $E\in\Eh$ the derivative $q_E'$ of $q_E$ is taken in the direction of $\tangent_E$. The interpretations of the components in $\Xcurl{h}$ and $\Xdiv{h}$ show that the definition of the face scalar curl $\CF$ and element divergence $\DT$ are based on integration-by-parts.

\subsubsection{Potential reconstructions and inner products}

For all $\sfP\in\Th\cup\Fh$, we define $\Xgrad{\sfP}$, $\Xcurl{\sfP}$ and $\Xdiv{\sfP}$ as the respective spaces obtained restricting the vectors of $\Xgrad{h}$, $\Xcurl{h}$ and $\Xdiv{h}$ to the edges (for $\Xgrad{h}$ and $\Xcurl{h}$) or faces (for $\Xdiv{h}$) of $\sfP$. We present here a reconstruction of face/element functions (potentials) based on the degrees of freedom in each local space; these potentials are then used to design $L^2$-like inner products in the discrete spaces.

For each $F\in\Fh$ and $T\in\Th$, the face (tangential) gradient $\cGF:\Xgrad{F}\to H_F$ (where $H_F$ is the plane spanned by $F$), scalar trace $\trF:\Xgrad{F}\to\Poly{1}(F)$, element gradient $\cGT:\Xgrad{T}\to\Real^3$ and scalar potential $\Pgrad:\Xgrad{T}\to\Poly{1}(T)$ are defined such that: for all $\underline{q}_F\in\Xgrad{F}$ and $\underline{q}_T\in\Xgrad{T}$,
\begin{alignat*}{4}
\cGF\underline{q}_F\cdot\bvec{\xi}={}&\frac{1}{|F|}\sum_{E\in\EF}\omega_{FE}\int_E q_E\bvec{\xi}\cdot\normal_{FE}&\qquad\forall\bvec{\xi}\in H_F,\\
\int_F \trF\underline{q}_F\DIV_F\bvec{v}_F={}& -\int_F \cGF\underline{q}_F\cdot\bvec{v}_F+\sum_{E\in\EF}\omega_{FE}\int_E q_E \bvec{v}_F\cdot\normal_{FE}&\qquad\forall \bvec{v}_F\in\cRoly{2}(F),\\
\cGT\underline{q}_T\cdot\bvec{\xi}={}&\frac{1}{|T|}\sum_{F\in\FT}\omega_{TF}\int_F \trF\underline{q}_F\,\bvec{\xi}\cdot\normal_{F}&\qquad\forall\bvec{\xi}\in\Real^3,\\
\int_T \Pgrad\underline{q}_T\DIV\bvec{v}_T={}& -\int_T \cGT\underline{q}_T\cdot\bvec{v}_T+\sum_{F\in\FT}\omega_{TF}\int_F \trF\underline{q}_F\, \bvec{v}_T\cdot\normal_{F}&\qquad\forall \bvec{v}_T\in\cRoly{2}(T).
\end{alignat*}
In the equations above, for $\sfP\in\{F,T\}$ we have set $\cRoly{2}(\sfP)=(\bvec{x}-\bvec{x}_{\sfP})\Poly{1}(\sfP)$, where $\bvec{x}_{\sfP}$ is a fixed point in $\sfP$, and the definition of the scalar trace and potential are justified by the fact that $\DIV_\sfP:\cRoly{2}(\sfP)\to\Poly{1}(\sfP)$ is an isomorphism.

For each $F\in\Fh$ and $T\in\Th$, the tangential trace $\trFt:\Xcurl{F}\to H_F$, element curl $\cCT:\Xcurl{T}\to\Real^3$ and vector potential $\Pcurl:\Xcurl{T}\to\Real^3$ are such that: for all $\uvec{v}_F\in\Xcurl{F}$ and all $\uvec{v}_T\in\Xcurl{T}$,
\begin{alignat*}{4}
\trFt\uvec{v}_F\cdot\VROT_F r_F={}& \frac{1}{|F|}\int_F \CF\uvec{v}_Fr_F+\frac{1}{|F|}\sum_{E\in\EF}\omega_{FE}\int_E v_Er_F&\quad\forall r_F\in\Poly{0,1}(F),\\
\cCT\uvec{v}_F\cdot\bvec{\xi}={}&\frac{1}{|T|}\sum_{F\in\FT}\omega_{TF}|F|\trFt\uvec{v}_F\cdot (\bvec{\xi}\times\normal_F)&\quad\forall\bvec{\xi}\in \Real^3,\\
\Pcurl\uvec{v}_T\cdot\CURL\bvec{w}_T={}&\frac{1}{|T|}\int_T\cCT\uvec{v}_T\cdot\bvec{w}_T-\frac{1}{|T|}\sum_{F\in\FT}\omega_{TF}\int_F\trFt\uvec{v}_F\cdot (\bvec{w}_T\times\normal_F)&\quad\forall\bvec{w}_T\in\cGoly{1}(T)
\end{alignat*}
where $\Poly{0,1}(F)=\{r_F\in\Poly{1}(F)\,:\,\int_F r_F=0\}$, while $\cGoly{1}(T)=(\bvec{x}-\bvec{x}_T)\times \Real^3$ (note that $\VROT_F:\Poly{0,1}(F)\to H_F$ and $\CURL:\cGoly{1}(T)\to\Real^3$ are isomorphisms).

For each $T\in\Th$, the vector potential reconstruction $\Pdiv:\Xdiv{T}\to\Real^3$ is such that, for all $\uvec{w}_T\in\Xdiv{T}$,
\begin{equation*}
\Pdiv\uvec{w}_T\cdot\GRAD r_T=-\frac{1}{|T|}\int_T\DT\uvec{w}_T r_T+\frac{1}{|T|}\sum_{F\in\FT}\omega_{TF}\int_F w_F r_T\qquad\forall r_T\in\Poly{0,1}(T).
\end{equation*}

In each DDR space, the discrete $L^2$-like inner product is then defined patching together local contributions built from the element potential reconstruction and stabilisation terms (the role of which is to control the differences between traces of the element potential and face/edge unknowns). So, for $\bullet\in\{\GRAD,\CURL,\DIV\}$ and $\underline{\mu}_h,\underline{\zeta}_h\in\Xbullet{h}$, we set
\[
(\underline{\mu}_h,\underline{\zeta}_h)_{\bullet,h}=\sum_{T\in\Th}\left(\int_T \Pbullet\underline{\mu}_T\,\Pbullet\underline{\zeta}_T
+\mathrm{s}_{\bullet,T}(\underline{\mu}_T,\underline{\zeta}_T)\right),
\]
where, for $T\in\Th$,
\begin{alignat*}{2}
\mathrm{s}_{\GRAD,T}(\underline{q}_T,\underline{r}_T)={}&\sum_{F\in\FT}h_F\int_F(\Pgrad\underline{q}_T-\trF\underline{q}_F)
(\Pgrad\underline{r}_T-\trF\underline{r}_F)\\
&+\sum_{E\in\ET}h_E^2\int_E(\Pgrad\underline{q}_T-q_E)(\Pgrad\underline{r}_T-r_E),\qquad\forall\underline{q}_h,\underline{r}_h\in\Xgrad{h},\\
\mathrm{s}_{\CURL,T}(\uvec{v}_T,\uvec{w}_T)={}&\sum_{F\in\FT}h_F\int_F((\Pcurl\uvec{v}_T)_{{\rm t}, F}-\trFt\uvec{v}_F)\cdot
((\Pcurl\uvec{w}_T)_{{\rm t}, F}-\trFt\uvec{w}_F)\\
&+\sum_{E\in\ET}h_E^2\int_E(\Pcurl\uvec{v}_T\cdot\tangent_E-v_E)(\Pcurl\uvec{w}_T\cdot\tangent_E-w_E),\quad\forall\uvec{v}_T,\uvec{w}_T\in\Xcurl{T}\\
\mathrm{s}_{\DIV,T}(\uvec{v}_T,\uvec{w}_T)={}&\sum_{F\in\FT}h_F\int_F(\Pdiv\uvec{v}_T\cdot\normal_F-v_F)
(\Pdiv\uvec{w}_T\cdot\normal_F - w_F),\quad\forall\uvec{v}_T,\uvec{w}_T\in\Xdiv{T},
\end{alignat*}
where $\bvec{\xi}_{{\rm t}, F}\coloneq\normal_F\times(\bvec{\xi}\times\normal_F)$ is the tangential component to $F$ of a vector $\bvec{\xi}\in\Real^3$.

\subsection{Lie algebra-valued DDR complex}\label{sec:LADDR}

There is a natural extension of the DDR complex to the de Rham complex with Lie algebra-valued function spaces, which we will call the LADDR (Lie algebra-valued DDR) complex.

\subsubsection{Spaces and extended operators} 

The spaces $\LaXgrad{h}$, $\LaXcurl{h}$ and $\LaXdiv{h}$, which are discrete versions of the function spaces $\LaHgrad{U}$, $\LaHcurl{U}$ and $\LaHdiv{U}$, have the following representations
\begin{alignat*}{2}
\LaXgrad{h}\coloneq&\Big\{
q_{\Eh}\in C^0(\Eh;\La)\st
  \text{$q_E\coloneq (q_{\Eh})_{|E}\in\Poly{1}(E;\La)$ for all $E\in\Eh$}
  \Big\},\\
\LaXcurl{h}\coloneq&\Big\{
  \uvec{v}_h
  =\big(
  (v_E)_{E\in\Eh}
  \big)\st\text{$v_E\in\La$ for all $E\in\Eh$}\Big\},\\
\LaXdiv{h}\coloneq&\Big\{
  \uvec{w}_h
  =\big((w_F)_{F\in\Fh}\big)\st
  \text{$w_F\in\La$ for all $F\in\Fh$}
  \Big\},
\end{alignat*}
where $C^0(\Eh;\La)$ (resp.~$\Poly{1}(E;\La)$) is the set of continuous functions $\Eh\to\La$ (resp.~linear functions $E\to\La$).

For all scalar operators $L:X\to Y$ in the DDR section, we denote the corresponding LADDR operators by $L^\La:X\otimes\La\to Y\otimes\La$, defined by setting $L^\La(v)=(Lv^I)\otimes e_I$, for all $v=v^I\otimes e_I\in X\otimes\La$. It can easily be checked using the linearity of $L$ that this definition does not depend on the choice of the basis $(e_I)_I$ in $\La$.

This definition and the cancellation property of the DDR complex ensure that the LADDR sequence
\begin{equation*}
  \begin{tikzcd}
    \Real\otimes\La\arrow{r}{\LaIgrad{h}} & \LaXgrad{h}\arrow{r}{\LauGh} & \LaXcurl{h}\arrow{r}{\LauCh} & \LaXdiv{h}\arrow{r}{\LaDh} & \Poly{0}(\Th)\otimes\La\arrow{r}{0} & \{0\}
  \end{tikzcd}
\end{equation*}
also forms a complex.

From the definitions of the DDR operators, the potential reconstructions then satisfy the corresponding integral characterisations in the Lie algebra setting. For example, for each $F\in\Fh$, the tangential trace $\LatrFt:\LaXcurl{F}\to H_F\otimes\La$ is such that, for all $\uvec{v}_F\in\LaXcurl{F}$,
\begin{equation}\label{eq:tang.tr}
\langle\LatrFt\uvec{v}_F,{\VROT_F} r_F\rangle= \frac{1}{|F|}\int_F \langle\LaCF\uvec{v}_F,r_F\rangle+\frac{1}{|F|}\sum_{E\in\EF}\omega_{FE}\int_E \langle v_E,r_F\rangle\quad\forall r_F\in\Poly{0,1}(F)\otimes\La.
\end{equation}
We will also make use of global potential reconstructions, defined piecewise on the mesh: for $\bullet\in\{\GRAD,\CURL\}$ and $x_h\in\LaXbullet{h}$, we define $\LaPbulleth x_h$ on $U$ by $(\LaPbulleth x_h)_{|T}=\LaPbullet x_T$ for all $T\in\Th$.

The inner product for $\bullet\in\{\GRAD,\CURL,\DIV\}$ and $\underline{\mu}_h,\underline{\zeta}_h\in\Xbullet{h}\otimes\La$ is defined by
\begin{equation}\label{eq:def.la.inner}
(\underline{\mu}_h,\underline{\zeta}_h)_{\bullet,\La,h}=(\underline{\mu}^I_h,\underline{\zeta}^J_h)_{\bullet,h}\langle e_I,e_J\rangle.
\end{equation}

\section{Schemes}\label{sec:schemes}

The schemes presented below are based on a (not necessarily uniform) partition of the time interval $[0,T]$, with timestep from $t^n$ to $t^{n+1}$ denoted by $\deltat^{n+\frac{1}{2}}\coloneq t^{n+1}-t^n$, and a given mesh $\Mh$ of the domain $U$ as in Section \ref{sec:DDR}.
The time-stepping is based on a $\theta$-scheme, with $\theta$ a fixed parameter in $[\frac12,1]$ ($\theta=\frac12$ corresponds to the Crank--Nicolson time stepping, $\theta=1$ to the implicit Euler time stepping). We also assume that discrete initial conditions $(\underline{\ymA}_h^0,\underline{\ymE}_h^0)\in (\LaXcurl{h})\times (\LaXcurl{h})$ are available; these can typically be chosen as
\begin{equation}\label{eq:ic.interpolated}
\underline{\ymA}_h^0=\LaIcurl{h}\ymA_0\,,\quad \underline{\ymE}_h^0=\LaIcurl{h}\ymE_0,
\end{equation}
but see also the discussion before Proposition \ref{prop:energy.lm.ym}.

\subsection{Maxwell equations}\label{sec:maxwell}

As an introductory example, we first consider the simpler case of the Maxwell equations of electromagnetism, which correspond to the Yang--Mills equations with Abelian Lie group $U(1)$. The Lie algebra $\mathfrak{u}(1)$ of the group $U(1)$ is commutative, and as such, has a trivial Lie bracket. The equations \eqref{eq:ym} therefore reduce to the following real-valued relations for the electric field $\ymE$ and the magnetic field $\ymB$:
\[
\partial_t\ymE-\CURL\ymB=0\quad\mbox{ and }\quad\partial_t\ymB+\CURL\ymE=0.
\]
The associated constraints are
\[
\DIV\ymE=0\quad\mbox{ and }\quad\DIV\ymB=0.
\]
The preservation of the first constraint, for example, follows by taking the divergence of the first equation:
$0=\DIV(\partial_t\ymE-\CURL\ymB)=\partial_t (\DIV\ymE)-\cancel{\DIV\CURL\ymB}$, where the cancellation comes from the complex property $\DIV\CURL=0$ of the de Rham complex; this shows that if $\DIV\ymE=0$ at $t=0$, then $\DIV\ymE=0$ at any time.

A weak formulation of the Maxwell problem, still under zero natural boundary conditions, is obtained by substituting $\ymB=\CURL\ymA$ into the evolution equation for $\ymE$, multiplying by a test function and integrating by parts: Find smooth enough functions $(\ymA,\ymE):[0,T]\to\Hcurl{U}\times\Hcurl{U}$ such that
\begin{subequations}\label{eq:mx.weak}
\begin{alignat}{4}
\partial_t\ymA={}&-\ymE,\\
\label{eq:mx.weak.1}
\int_U\langle\partial_t\ymE,\bvec{v}\rangle={}&\int_U\langle\CURL\ymA,\CURL \bvec{v}\rangle,&\qquad\forall\bvec{v}\in\Hcurl{U}.
\end{alignat}
\end{subequations}
Setting $\bvec{v}=\GRAD q$ for an arbitrary $q\in H^1(U)$, we see that \eqref{eq:mx.weak.1} preserves the following weak version of ``$\DIV\ymE=0$'' (together with the boundary condition $\ymE\cdot\normal=0$ from \eqref{eq:bcs}), in the sense that if this equation holds at $t=0$ it holds for any time:
\begin{equation*}
\int_U\langle\ymE,\GRAD q\rangle=0,\qquad\forall q\in H^1(U).
\end{equation*}

The DDR-based discrete scheme for \eqref{eq:mx.weak} with $\theta$-scheme time stepping is: Find families $(\underline{\ymA}_h^n)_n$, $(\underline{\ymE}_h^n)_n$, where $(\underline{\ymA}_h^n,\underline{\ymE}_h^n)\in\Xcurl{h}\times\Xcurl{h}$, such that for all $n$,
\begin{subequations}\label{eq:mx.scheme}
\begin{alignat}{4}
\delta_t^{n+1}\underline{\ymA}_h&=-\underline{\ymE}_h^{n+\theta},\\
(\delta_t^{n+1}\underline{\ymE}_h,\uvec{v}_h)_{\CURL,h}&=(\uCh\underline{\ymA}_h^{n+\theta},\uCh\uvec{v}_h)_{\CURL,h},&\qquad\forall\uvec{v}_h\in\Xcurl{h},\label{eq:mx.scheme.2}
\end{alignat}
\end{subequations}
where, for a family $(\bvec{Z}^n)_n$, we define
\begin{equation*}
\delta_t^{n+1}\bvec{Z}=\frac{1}{\deltat^{n+\frac{1}{2}}}(\bvec{Z}^{n+1}-\bvec{Z}^n),\quad
\bvec{Z}^{n+\theta}=\theta\bvec{Z}^{n+1}+(1-\theta)\bvec{Z}^n.
\end{equation*}
As for the continuous model, plugging $\uvec{v}_h=\uGh\underline{q}_h$ in \eqref{eq:mx.scheme.2} for a generic $\underline{q}_h\in\Xgrad{h}$ and using the complex property $\uCh\uGh=\uvec{0}_h$ of the DDR sequence, we see that the quantity $(\underline{\ymE}_h^n,\uGh \underline{q}_h)_{\CURL,h}$ is preserved throughout the evolution.

\subsection{Yang--Mills equations}\label{sec:ym}

\subsubsection{Weak formulation}

For most Lie groups, the Lie bracket of the Lie algebra does not vanish as in electromagnetism. We get a weak formulation of the Yang--Mills problem, with zero natural boundary conditions, by again multiplying the evolution equation \eqref{eq:ym.ev.E} by a test function, and using integration by parts to move the $\CURL$. To obtain the weak formulation, we also use the following property: for all Lie algebra-valued vector fields $(\bvec{u},\bvec{v},\bvec{w}$), using the Ad-invariance \eqref{eq:ad.inv} and the symmetry properties of the volume form, we have
\begin{alignat*}{4}
\int_U\langle\ebkt{\bvec{u}}{\bvec{v}},\bvec{w}\rangle=\int_U\langle\varepsilon\indices{^\alpha_{\mu\nu}}[\bvec{u}^\mu,\bvec{v}^\nu],\bvec{w}_\alpha\rangle&=\int_U\langle\bvec{u}^\mu,\varepsilon\indices{^\alpha_{\mu\nu}}[\bvec{v}^\nu,\bvec{w}_\alpha]\rangle\\
&=\int_U\langle\bvec{u}_\mu,\varepsilon\indices{^\mu_{\nu\alpha}}[\bvec{v}^\nu,\bvec{w}^\alpha]\rangle=\int_U\langle\bvec{u},\ebkt{\bvec{v}}{\bvec{w}}\rangle.
\end{alignat*}

The weak form of the Yang--Mills equations \eqref{eq:ym} is then: Find smooth enough functions $(\ymA,\ymE):[0,T]\to(\LaHcurl{U})\times(\LaHcurl{U}$) such that, setting $\ymB=\CURL\ymA+\frac{1}{2}\ebkt{\ymA}{\ymA}$,
\begin{subequations}\label{eq:ym.weak}
\begin{alignat}{4}
\partial_t\ymA={}&-\ymE,\label{eq:ym.weak.1}\\
\label{eq:ym.weak.2}
\int_U\langle\partial_t\ymE,\bvec{v}\rangle={}&\int_U\langle\ymB,\CURL \bvec{v}+\ebkt{\ymA}{\bvec{v}}\rangle,&\qquad\forall\bvec{v}\in\LaHcurl{U}.
\end{alignat}
\end{subequations}
The constraint \eqref{eq:ym.const.E} can be recast in weak form as
\begin{equation}\label{eq:constraint.ym}
\int_U\langle\ymE,\GRAD q+[\ymA,q]\rangle=0,\qquad\forall q\in \LaHgrad{U},
\end{equation}
where we use again integration by parts for the derivative and \eqref{eq:ad.inv} to see that for Lie algebra-valued $(q,\bvec{v},\bvec{w})$,
\begin{equation*}
\int_U\langle\dbkt{\bvec{v}}{\bvec{w}},q\rangle=\int_U\langle[\bvec{v}^\mu,\bvec{w}_\mu],q\rangle=\int_U\langle\bvec{v}^\mu,[\bvec{w}_\mu,q]\rangle=\int_U\langle\bvec{v},[\bvec{w},q]\rangle.
\end{equation*}

Let us recall that these weak constraints are preserved in time.

\begin{proposition}[Preservation of constraint for the weak Yang--Mills equations]\label{prop:ym.preservation.constraint}
If $(\ymA,\ymE)$ solve \eqref{eq:ym.weak} then for all $q\in \LaHgrad{U}$ the quantity
\begin{equation}\label{def:weak.constraint}
\int_U\langle\ymE,\GRAD q+[\ymA,q]\rangle
\end{equation}
is constant in time. In particular, if the constraint \eqref{eq:constraint.ym} is satisfied at $t=0$, then it is satisfied for all $t\in [0,T]$.
\end{proposition}
 
\begin{proof}
Fixing $q$ (constant in time) and differentiating with respect to time, we find
\begin{equation}\label{eq:dt.const}
\int_U\partial_t\langle\ymE,\GRAD q+[\ymA,q]\rangle=\int_U\langle\partial_t\ymE,\GRAD q+[\ymA,q]\rangle + \int_U\langle\ymE,[\partial_t\ymA,q]\rangle.
\end{equation}
The second integral on the right is zero due to \eqref{eq:ym.weak.1} and the result of the Ad-invariance \eqref{eq:ad.inv} of the inner product on the Lie algebra. Substituting $\bvec{v}=\GRAD q+[\ymA,q]$ into \eqref{eq:ym.weak.2},
\begin{equation}\label{eq:alg.sub}
\int_U\langle\partial_t\ymE,\GRAD q+[\ymA,q]\rangle=\int_U\langle\ymB,\CURL(\GRAD q+[\ymA,q])+\ebkt{\ymA}{\GRAD q+[\ymA,q]}\rangle.
\end{equation}
The curl of gradient term disappears on the right, and we expand the rest in vector components to get
\begin{alignat}{2}
\langle\ymB,{}&\CURL([\ymA,q])+\ebkt{\ymA}{\GRAD q+[\ymA,q]}\rangle\nonumber\\
={}&\langle\ymB^\alpha,\varepsilon\indices{_{\alpha\mu\nu}}\partial^\mu[\ymA^{\nu},q]+\varepsilon\indices{_{\alpha\mu\nu}}[\ymA^\mu,\partial^\nu q+[\ymA^\nu,q]]\rangle\nonumber\\
={}&\langle\ymB^\alpha,\varepsilon\indices{_{\alpha\mu\nu}}[\partial^\mu\ymA^{\nu},q]+
\cancel{\varepsilon\indices{_{\alpha\mu\nu}}[\ymA^{\nu},\partial^\mu q]
+\varepsilon\indices{_{\alpha\mu\nu}}[\ymA^\mu,\partial^\nu q]}+\varepsilon\indices{_{\alpha\mu\nu}}[\ymA^\mu,[\ymA^\nu,q]]\rangle\nonumber\\
={}&\langle\ymB^\alpha,\left[\varepsilon\indices{_{\alpha\mu\nu}}\partial_\mu\ymA^{\nu}+\frac{1}{2}\varepsilon\indices{_{\alpha\mu\nu}}[\ymA^\mu,\ymA^\nu],q\right]\rangle\nonumber\\
={}&\langle\ymB,[\ymB,q]\rangle,
\label{eq:alg.cancel}
\end{alignat}
where the cancellation is justified by the antisymmetry of $\varepsilon$, and we have used the Jacobi identity to write
\[
\varepsilon\indices{_{\alpha\mu\nu}}[\ymA^\mu,[\ymA^\nu,q]]=\frac{1}{2}\varepsilon\indices{_{\alpha\mu\nu}}[[\ymA^\mu,\ymA^\nu],q]=\left[\frac{1}{2}\varepsilon\indices{_{\alpha\mu\nu}}[\ymA^\mu,\ymA^\nu],q\right].
\]
Since $\langle\ymB,[\ymB,q]\rangle$ also vanishes due to \eqref{eq:ad.inv}, this establishes the preservation in time of \eqref{def:weak.constraint}.
\end{proof}

\subsubsection{Constrained weak formulation}

When considering the discretisation of \eqref{eq:ym.weak}, we see that $\Xcurl{h}\otimes \La$ is a natural space for the electric field $\ymE$ and potential $\ymA$; then the discrete version of $\CURL\ymA$ is expected to belong to $\Xdiv{h}\otimes\La$, which suggests that this would be the discrete space for the magnetic field $\ymB$. However, we then need to understand how to discretise the non-linear term $\ebkt{\ymA}{\ymA}$  into an element of $\Xdiv{h}\otimes\La$. In addition, it needs to be done in a way so that we can reproduce the proof of Proposition \ref{prop:ym.preservation.constraint} in some sense. If the discrete version of $\CURL\GRAD=\bvec{0}$ is ensured by the complex property of the LADDR sequence, it is not clear how to reproduce at the discrete level the algebraic cancellations of \eqref{eq:alg.cancel}. For this reason, we consider an alternate formulation which directly imposes, through the usage of Lagrange multipliers, the first term of the right-hand side of \eqref{eq:dt.const} to be zero. Then reproducing the Ad-invariance and antisymmetry of the bracket term at the discrete level is sufficient for preserving the discrete constraint. The Lagrange multiplier approach of Yang--Mills equations has already been considered in \cite{Christiansen.Winther:06} for conforming numerical approximations, with numerical tests shown for 2-dimensional standard Finite Elements on triangles.

The constrained weak formulation is: Find $(\ymA,\ymE,\lambda):[0,T]\to(\LaHcurl{U})\times(\LaHcurl{U})\times (\LaHgrad{U})$ such that
\begin{subequations}\label{eq:lm.ym}
\begin{alignat}{4}
\partial_t\ymA={}&-\ymE,\\
\label{eq:lm.ym.1}
\int_U\langle\partial_t\ymE,\bvec{v}\rangle+\int_U\langle\GRAD\lambda+[\ymA,\lambda],\bvec{v}\rangle={}&\int_U\langle\ymB,\CURL \bvec{v}+\ebkt{\ymA}{\bvec{v}}\rangle,&\qquad\forall\bvec{v}\in\LaHcurl{U},\\
\label{eq:lm.ym.2}
\int_U\langle\partial_t\ymE,\GRAD q+[\ymA,q]\rangle={}&0,&\qquad\forall q\in \LaHgrad{U}.
\end{alignat}
\end{subequations}

\begin{remark}[Alternative approaches]
In conforming schemes, based on finite dimensional subspaces of $\Hcurl{\Omega}$, the algebraic manipulations \eqref{eq:alg.cancel} hold at the discrete level. The issue arises, however, in the substitution $\bvec{v}=\GRAD q+[\ymA,q]$ in Equation \eqref{eq:ym.weak.2}, since the non-linear term generally does not belong to the space of test functions (in polynomial spaces, the degree of $[\ymA,q]$ is the sum of the degrees of $\ymA$ and $q$). This was noted in \cite{Christiansen.Winther:06}, where the proof of Proposition \ref{prop:ym.preservation.constraint} could only be reproduced at the discrete level if $q$ is taken to be in the space of constant Lie algebra-valued functions, which only implies conservation of charge over the entire domain. A solution to this conundrum is put forward in \cite{Berchenko-Kogan.Stern:21} where a low-order hybridized scheme is proposed, based on a discretisation of \eqref{eq:ym.weak} using completely discontinuous spaces, and complemented with continuous conditions enforced through the addition of Lagrange multipliers. Similar methods have been considered, e.g. in \cite{rhebergen.wells:2018:hybridizable,botti.massa:2021:hho}, to recover pointwise divergence-free (and conservative) approximations of the velocity in the Navier--Stokes equations. However, this means that only the lowest order moment of the constraint on each element can be preserved, as the hybridized method would face the same issue of invalid test functions if higher degree was considered; on the contrary, the higher-order version of the LADDR scheme (see \eqref{eq:ym.lm.scheme.k} below) preserves high-order, if less local, moments of the constraint. We also note that the imposition of the continuity conditions in the hybridized scheme requires to add at least two Lie algebra unknowns per edge (in 2D) or face (in 3D); on the contrary, the constrained formulation \eqref{eq:lm.ym} only adds one Lie algebra unknown per vertex, and therefore leads to a much leaner scheme (especially in 3D).
\end{remark}

\begin{lemma}[Relationship between weak and constrained weak formulation]
If $(\ymA,\ymE)$ solves \eqref{eq:ym.weak}, then $(\ymA,\ymE,0)$ solves \eqref{eq:lm.ym}. Conversely, if $(\ymA,\ymE,\lambda)$ solves \eqref{eq:lm.ym} then $(\ymA,\ymE)$ solves \eqref{eq:ym.weak}.
\end{lemma}

\begin{proof}
Let $(\ymA,\ymE)$ be a solution to \eqref{eq:ym.weak} and set $\lambda=0$. Then the second term in the left-hand side of \eqref{eq:lm.ym.1} vanishes and this equation is simply \eqref{eq:ym.weak.2}, while the extra equation \eqref{eq:lm.ym.2} follows from \eqref{eq:alg.sub} and \eqref{eq:alg.cancel}. Conversely, if $(\ymA,\ymE,\lambda)$ solves \eqref{eq:lm.ym} then, choosing $\bvec{v}=\GRAD\lambda+[\ymA,\lambda]$ in \eqref{eq:lm.ym.1} and using \eqref{eq:lm.ym.2} with $q=\lambda$, as well as the calculation \eqref{eq:alg.cancel} for the right side, we have $\GRAD\lambda+[\ymA,\lambda]=0$. Hence, \eqref{eq:lm.ym.1} reduces to \eqref{eq:ym.weak.2}.
\end{proof}

\begin{remark}[About the Lagrange multiplier]\label{rem:lagrange}
There is no certainty that \eqref{eq:lm.ym} uniquely determines $\lambda$. \cite{Christiansen.Winther:06} claims that, since this system is formally equivalent to \eqref{eq:ym.weak}, it should lead to a zero Lagrange multiplier. However, what the argument above actually shows is that the quantity $\GRAD\lambda+[\ymA,\lambda]$ vanishes; it is unclear how that would necessarily mean that $\lambda$ itself vanishes; in the case $\ymA=0$ for example, an additional condition on $\lambda$ -- e.g. fixing its average -- must be imposed to obtain its uniqueness. Proposition 3.4 in \cite{Christiansen.Winther:06} discusses, in the case of a conforming approximation $X_h^0$ of $H^1$, the question of uniform stability (for small enough $h$) of the Lagrange multiplier when seeking that multiplier in $X_h^0\cap S$, where $S$ is a subspace of $H^1$ that has a zero intersection with the kernel of $\GRAD\cdot+[\ymA,\cdot]$. Writing the discretisation \eqref{eq:ym.lm.scheme} of the constrained problem \eqref{eq:lm.ym} using, in the conforming setting, $X_h^0\cap S$ instead of $X_h^0$ would restrain the choice of test functions in \eqref{eq:ym.lm.scheme.3} to $X_h^0\cap S$ which would, in turn, only allow to prove the preservation of constraint for such test functions (see the proofs of \cite[Proposition 3]{Christiansen.Winther:06} and Proposition \ref{prop:constraint.preservation} below); note also that, finding $S$ is not an easy task, all the more as, to maintain as much as possible the constraint, one would have to take $S$ such that $X_h^0\cap S$ remains as large as possible.

Our numerical experiments actually seem to indicate that, in the discrete setting at least (and without imposing any additional constraint on the Lagrange multiplier), there is a solution to \eqref{eq:lm.ym} but it may not be unique -- see the discussion in Section \ref{sec:tests}. We also checked that imposing, for example, the average of $\lambda$ leads to a numerical solution that is unique, but breaks the full constraint preservation ($\mathfrak C^n$ defined in \eqref{discrete.const} does not remain constant for all $\underline{q}_h$). We however believe that the non-uniqueness issue only affects the Lagrange multiplier (which does not have any genuine physical meaning), and that the quantities of importance $\ymE,\ymA$ remain uniquely determined by the constrained equations (in particular because of the energy estimate \eqref{eq:lm.ym.energy} below).
\end{remark}

\subsubsection{Discretisation of the Lie brackets}

Two kinds of Lie brackets are involved in \eqref{eq:lm.ym}: one acting on $(\LaHcurl{U})\times(\LaHcurl{U})$, the other on $(\LaHcurl{U})\times(\LaHgrad{U})$. The requirements (for stability and constraint preservation) on the discretisations of these brackets are different, so we will employ different approaches for each of them. The brackets on $(\LaHcurl{U})\times (\LaHcurl{U})$ appear in the inner products with differential operators, and as well as in the definition of $\ymB$. Therefore it is convenient to explicitly construct a discrete version of this bracket as an element of $\LaXdiv{h}$, so that we can make use of the properties of our discrete differential operators and inner products in calculations. The brackets on $(\LaHcurl{U})\times(\LaHgrad{U})$ on the other hand are not mixed with derivatives. The discrete version needs to instead reproduce the Ad-invariance \eqref{eq:ad.inv} with the (discrete) inner product, which is much more straightforward.

For given Lie algebra-valued vector fields $\bvec{v},\bvec{w}$, the discretisation of $\ebkt{\bvec{v}}{\bvec{w}}$ in $\LaXdiv{h}$ relies on the fact that, for any $F\in\Fh$, $\ebkt{\bvec{v}}{\bvec{w}}\cdot\normal_F$ depends only on the tangential components $\bvec{v}_{{\rm t},F},\bvec{w}_{{\rm t},F}$ of $\bvec{v},\bvec{w}$ to the face $F$. This is easily seen using the decomposition $\bvec{\xi} = (\bvec{\xi}\cdot\normal_F)\normal_F + \bvec{\xi}_{{\rm t},F}$ of a vector into its normal and tangential parts relative to $F$. Applying this decomposition to $\bvec{v}$ and $\bvec{w}$, $\ebkt{\bvec{v}}{\bvec{w}}\cdot\normal_F$ reduces to $\ebkt{\bvec{v}_{{\rm t},F}}{\bvec{w}_{{\rm t},F}}\cdot\normal_F$ due to the calculation
\[
\varepsilon\indices{^\alpha_{\mu\nu}}[\bvec{v}^\mu,\bvec{w}^\nu](\normal_F)_\alpha=\varepsilon\indices{^\alpha_{\mu\nu}}[(\bvec{v}\cdot\normal_F)\normal_F^\mu + \bvec{v}_{{\rm t}, F}^\mu,(\bvec{w}\cdot\normal_F)\normal_F^\nu + \bvec{w}_{{\rm t}, F}^\nu](\normal_F)_\alpha=\varepsilon\indices{^\alpha_{\mu\nu}}[\bvec{v}_{{\rm t}, F}^\mu,\bvec{w}_{{\rm t}, F}^\nu](\normal_F)_\alpha,
\]
where we have used $\varepsilon\indices{^\alpha_{\mu\nu}}\normal_F^\mu(\normal_F)_\alpha=\varepsilon\indices{^\alpha_{\mu\nu}}\normal_F^\nu(\normal_F)_\alpha=0$ by antisymmetry of $\varepsilon$, leaving us with just the final term after expanding the bracket. 

The tangential trace $\LatrFt$ (see \eqref{eq:tang.tr}) provides, from elements of $\LaXcurl{h}$, discrete versions of $\bvec{v}_{{\rm t}, F}$, $\bvec{w}_{{\rm t}, F}$. This discrete tangential trace is moreover consistent when applied to the interpolate of constant fields. Since the values of $\LatrFt$ depend only on the values of the elements of $\LaXcurl{h}$ on the edges of $F$, we are able to create a globally well defined element of $\LaXdiv{h}$. This leads to the following definition of the symmetric bilinear form $\ebkttr{\cdot}{\cdot}:(\LaXcurl{h})\times(\LaXcurl{h})\to\LaXdiv{h}$: For any $\uvec{v}_h$, $\uvec{w}_h\in\LaXcurl{h}$,
\begin{equation}\label{eq:def.brack.xdiv}
(\ebkttr{\uvec{v}_h}{\uvec{w}_h})_F=\ebkt{\LatrFt\uvec{v}_F}{\LatrFt\uvec{w}_F}\cdot\normal_F,\qquad\forall F\in\Fh.
\end{equation}

The discretisation of integral terms, in \eqref{eq:lm.ym}, involving the bracket in $(\LaHcurl{U})\times(\LaHgrad{U})$ is a bit more straightforward. We simply use the global potential reconstructions:
\[
\int_U \langle\bvec{v},[\bvec{w},q]\rangle\leadsto \int_U\langle\LaPcurlh\uvec{v}_h,[\LaPcurlh\uvec{w}_h,\LaPgradh\underline{q}_h]\rangle.
\]
We note that, in this discretisation of the coupling between potential and Lagrange multiplier, and contrary to the discretisations \eqref{eq:def.brack.xdiv} and \eqref{eq:ym.lm.scheme.2} of the coupling between magnetic field and potential, the Lie algebra inner product and bracket themselves are not discretised (only the functions are, through the usage of discrete vectors $\uvec{v}_h$, $\uvec{w}_h$, $\underline{q}_h$ and the potential reconstructions); this choice is made to ensure that, by \eqref{eq:ad.inv}, the Ad-invariance holds for this discretisation; in particular, if $\uvec{v}_h=\uvec{w}_h$, this discretisation vanishes.

\subsubsection{Unconstrained scheme}

The $\theta$-scheme for \eqref{eq:ym.weak} is: Find families $(\underline{\ymA}_h^n)_n$, $(\underline{\ymE}_h^n)_n$ such that for all $n$, $(\underline{\ymA}_h^n,\underline{\ymE}_h^n) \in(\LaXcurl{h})\times(\LaXcurl{h})$ and
\begin{subequations}\label{eq:unconst}
\begin{alignat}{4}\label{eq:unconst.1}
\delta_t^{n+1}\underline{\ymA}_h&=-\underline{\ymE}_h^{n+\theta},\\
\label{eq:unconst.2}
(\delta_t^{n+1}\underline{\ymE}_h,\uvec{v}_h)_{\CURL,\La,h}&=(\underline{\ymB}_h^{n+\theta},\LauCh\uvec{v}_h+\ebkttr{\underline{\ymA}_h^{n+\frac{1}{2}}}{\uvec{v}_h})_{\DIV,\La,h},&\qquad\forall\uvec{v}_h\in\LaXcurl{h},
\end{alignat}
\end{subequations}
where 
\begin{equation}\label{eq:def.ymBh}
\underline{\ymB}_h^n=\LauCh\underline{\ymA}_h^n+\frac12\ebkttr{\underline{\ymA}_h^n}{\underline{\ymA}_h^n}.
\end{equation}
This scheme does not guarantee the preservation of the constraint in any obvious discrete form, but we have decay of energy for any $\theta$, and preservation for $\theta=\frac12$ (Crank--Nicolson time-stepping).
\begin{proposition}[Energy conservation]\label{prop:energy.conservation}
The scheme \eqref{eq:unconst} has energy decay in the sense that, for all $n$,
\begin{equation*}
\frac{1}{2}\norm[\CURL,\La,h]{\underline{\ymE}_h^n}^2+\frac{1}{2}\norm[\DIV,\La,h]{\underline{\ymB}_h^n}^2\leq\frac{1}{2}\norm[\CURL,\La,h]{\underline{\ymE}_h^{n+1}}^2+\frac{1}{2}\norm[\DIV,\La,h]{\underline{\ymB}_h^{n+1}}^2,
\end{equation*}
with energy conservation (equality) when $\theta=\frac{1}{2}$.
\end{proposition}
\begin{proof}
Choosing $\bvec{v}=\underline{\ymE}_h^{n+\theta}$ in \eqref{eq:unconst.2}, and multiplying both sides by $\deltat^{n+\frac{1}{2}}$, we have on the left:
\begin{alignat}{4}\nonumber
(\underline{\ymE}_h^{n+1}-\underline{\ymE}_h^n,\underline{\ymE}_h^{n+\theta})_{\CURL,\La,h}{}&=(\underline{\ymE}_h^{n+1}-\underline{\ymE}_h^n,\frac{1}{2}\underline{\ymE}_h^{n+1}+\frac{1}{2}\underline{\ymE}_h^n+(\theta-\frac{1}{2})\underline{\ymE}_h^{n+1}-(\theta-\frac{1}{2})\underline{\ymE}_h^n)_{\CURL,\La,h}\\ \label{eq:energy.norm.calc}
&=\frac{1}{2}\norm[\CURL,\La,h]{\underline{\ymE}_h^{n+1}}^2-\frac{1}{2}\norm[\CURL,\La,h]{\underline{\ymE}_h^n}^2+(\theta-\frac{1}{2})\norm[\CURL,\La,h]{\underline{\ymE}_h^{n+1}-\underline{\ymE}_h^n}^2.
\end{alignat}
On the right, using $\deltat^{n+\frac{1}{2}}\underline{\ymE}_h^{n+\theta}=\underline{\ymA}_h^n-\underline{\ymA}_h^{n+1}$ by \eqref{eq:unconst.1}, and the symmetry of $\ebkttr{\cdot}{\cdot}$ (coming from its definition \eqref{eq:def.brack.xdiv} together with the anti-symmetries of $\varepsilon$ and of the Lie bracket),
\begin{alignat*}{4}
\deltat^{n+\frac{1}{2}}{}&(\underline{\ymB}_h^{n+\theta},\LauCh\underline{\ymE}_h^{n+\theta}+\ebkttr{\underline{\ymA}_h^{n+\frac{1}{2}}}{\underline{\ymE}_h^{n+\theta}})_{\DIV,\La,h}\\
&=(\underline{\ymB}_h^{n+\theta},\LauCh(\underline{\ymA}_h^n-\underline{\ymA}_h^{n+1})+\frac{1}{2}\ebkttr{\underline{\ymA}_h^n+\underline{\ymA}_h^{n+1}}{\underline{\ymA}_h^n-\underline{\ymA}_h^{n+1}})_{\DIV,\La,h}\\
&=(\underline{\ymB}_h^{n+\theta},\LauCh\underline{\ymA}_h^n+\frac{1}{2}\ebkttr{\underline{\ymA}_h^n}{\underline{\ymA}_h^n}-\LauCh\underline{\ymA}_h^{n+1}-\frac{1}{2}\ebkttr{\underline{\ymA}_h^{n+1}}{\underline{\ymA}_h^{n+1}})_{\DIV,\La,h}\\
&=(\underline{\ymB}_h^{n+\theta},\underline{\ymB}_h^n-\underline{\ymB}_h^{n+1})_{\DIV,\La,h}.
\end{alignat*}
Using the same calculation as in \eqref{eq:energy.norm.calc}, we obtain 
\begin{equation*}
(\underline{\ymB}_h^{n+\theta},\underline{\ymB}_h^n-\underline{\ymB}_h^{n+1})_{\DIV,\La,h}=\frac{1}{2}\norm[\DIV,\La,h]{\underline{\ymB}_h^n}^2-\frac{1}{2}\norm[\DIV,\La,h]{\underline{\ymB}_h^{n+1}}^2-(\theta-\frac{1}{2})\norm[\DIV,\La,h]{\underline{\ymB}_h^{n+1}-\underline{\ymB}_h^n}^2,
\end{equation*}
Regrouping the terms, we get
\begin{alignat*}{4}
\frac{1}{2}\norm[\CURL,\La,h]{\underline{\ymE}_h^{n+1}}^2+\frac{1}{2}\norm[\DIV,\La,h]{\underline{\ymB}_h^{n+1}}^2={}&\frac{1}{2}\norm[\CURL,\La,h]{\underline{\ymE}_h^n}^2+\frac{1}{2}\norm[\DIV,\La,h]{\underline{\ymB}_h^n}^2\\
&-(\theta-\frac{1}{2})(\norm[\DIV,\La,h]{\underline{\ymB}_h^{n+1}-\underline{\ymB}_h^n}^2+\norm[\CURL,\La,h]{\underline{\ymE}_h^{n+1}-\underline{\ymE}_h^n}^2),
\end{alignat*}
which concludes the proof since $\theta\geq\frac{1}{2}$.
\end{proof}

\subsubsection{Constrained scheme}

The $\theta$-scheme for the constrained formulation \eqref{eq:lm.ym} is: Find families $(\underline{\ymA}_h^n)_n$, $(\underline{\ymE}_h^n)_n$, $(\underline{\lambda}_h^n)_n$ such that for all $n$, $(\underline{\ymA}_h^n,\underline{\ymE}_h^n,\underline{\lambda}_h^n)\in(\LaXcurl{h})\times(\LaXcurl{h})\times(\LaXgrad{h})$ and
\begin{subequations}\label{eq:ym.lm.scheme}
\begin{alignat}{4}
\delta_t^{n+1}\underline{\ymA}_h&=-\underline{\ymE}_h^{n+\theta},\label{eq:ym.lm.scheme.1}\\ 
(\delta_t^{n+1}\underline{\ymE}_h,\uvec{v}_h)_{\CURL,\La,h}+{}&(\LauGh{\underline{\lambda}_h^{n+1}},\uvec{v}_h)_{\CURL,\La,h}+\int_U\langle[\LaPcurlh\underline{\ymA}_h^{n+\theta},\LaPgradh\underline{\lambda}_h^{n+1}],\LaPcurlh\uvec{v}_h\rangle\nonumber\\
&=(\underline{\ymB}_h^{n+\theta},\LauCh\uvec{v}_h+\ebkttr{\underline{\ymA}_h^{n+\frac{1}{2}}}{\uvec{v}_h})_{\DIV,\La,h},\qquad\forall\uvec{v}_h\in\LaXcurl{h},\label{eq:ym.lm.scheme.2}\\
(\delta_t^{n+1}\underline{\ymE}_h,\LauGh\underline{q}_h)_{\CURL,\La,h}{}&+\int_U\langle\LaPcurlh(\delta_t^{n+1}\underline{\ymE}_h),[\LaPcurlh\underline{\ymA}_h^{n+1-\theta},\LaPgradh\underline{q}_h]\rangle\nonumber\\
&=0,\qquad\forall\underline{q}_h\in\LaXgrad{h},\label{eq:ym.lm.scheme.3}
\end{alignat}
\end{subequations}
where $\underline{\ymB}_h^n$ is still given by \eqref{eq:def.ymBh}.

We define the discrete constraint functional $\mathfrak C^n:\LaXgrad{h}\to\Real$ by
\begin{equation}\label{discrete.const}
\mathfrak C^n(\underline{q}_h)\coloneq(\underline{\ymE}_h^n,\LauGh\underline{q}_h)_{\CURL,\La,h}+\int_U\langle\LaPcurlh\underline{\ymE}_h^n,[\LaPcurlh\underline{\ymA}_h^n,\LaPgradh\underline{q}_h]\rangle\qquad\forall\underline{q}_h\in\LaXgrad{h}.
\end{equation}

\begin{proposition}[Constraint preservation]\label{prop:constraint.preservation}
If $(\underline{\ymA}_h^n,\underline{\ymE}_h^n,\underline{\lambda}_h^n)$ solve \eqref{eq:ym.lm.scheme} then, for all $\underline{q}_h\in\LaXgrad{h}$, the quantity $\mathfrak C^n(\underline{q}_h)$ is independent of $n$.
\end{proposition}

\begin{proof}
Using the same idea as for the continuous case (see the proof of Proposition \ref{prop:ym.preservation.constraint}), we fix $\underline{q}_h$ (independent of $n$), and take the discrete derivative of \eqref{discrete.const}. We can expand any bilinear product using the formulas $\delta_t^{n+1}(u,v)=(\delta_t^{n+1}u,v^{n+1})+(u^n,\delta_t^{n+1}v)=(\delta_t^{n+1}u,v^n)+(u^{n+1},\delta_t^{n+1}v)$ to get
\begin{alignat*}{4}
\delta_t^{n+1}\mathfrak C^n(\underline{q}_h)={}&\delta_t^{n+1}(\underline{\ymE}_h,\LauGh\underline{q}_h)_{\CURL,\La,h}+(1-\theta)\delta_t^{n+1}\int_U\langle\LaPcurlh\underline{\ymE}_h,[\LaPcurlh\underline{\ymA}_h,\LaPgradh\underline{q}_h]\rangle\\
&+\theta\delta_t^{n+1}\int_U\langle\LaPcurlh\underline{\ymE}_h,[\LaPcurlh\underline{\ymA}_h,\LaPgradh\underline{q}_h]\rangle\\
={}&(\delta_t^{n+1}\underline{\ymE}_h,\LauGh\underline{q}_h)_{\CURL,\La,h}+(1-\theta)\int_U\langle\LaPcurlh(\delta_t^{n+1}\underline{\ymE}_h),[\LaPcurlh\underline{\ymA}_h^{n+1},\LaPgradh\underline{q}_h]\rangle\\
&+(1-\theta)\int_U\langle\LaPcurlh\underline{\ymE}_h^n,[\LaPcurlh(\delta_t^{n+1}\underline{\ymA}_h),\LaPgradh\underline{q}_h]\rangle\\
&+\theta\int_U\langle\LaPcurlh(\delta_t^{n+1}\underline{\ymE}_h),[\LaPcurlh\underline{\ymA}_h^n,\LaPgradh\underline{q}_h]\rangle\\
&+\theta\int_U\langle\LaPcurlh(\underline{\ymE}_h^{n+1}),[\LaPcurlh(\delta_t^{n+1}\underline{\ymA}_h),\LaPgradh\underline{q}_h]\rangle\\
={}&(\delta_t^{n+1}\underline{\ymE}_h,\LauGh\underline{q}_h)_{\CURL,\La,h}+ \int_U\langle\LaPcurlh(\delta_t^{n+1}\underline{\ymE}_h),[\LaPcurlh\underline{\ymA}_h^{n+1-\theta},\LaPgradh\underline{q}_h]\rangle\\
&+\int_U\langle\LaPcurlh(\underline{\ymE}_h^{n+\theta}),[\LaPcurlh(\delta_t^{n+1}\underline{\ymA}_h),\LaPgradh\underline{q}_h]\rangle=0,
\end{alignat*}
where after collecting the like terms in the second equality, the use of \eqref{eq:ym.lm.scheme.3} takes care of the first line of the third equality, and \eqref{eq:ym.lm.scheme.1} along with the Ad-invariance deals with the second line. This conclu\-des the proof that \eqref{discrete.const} is independent of $n$.
\end{proof}

Even if the initial conditions $(\ymA_0,\ymE_0)$ of the Yang--Mills equations satisfy the constraint \eqref{eq:ym.const.E}, it is not expected that the discrete initial conditions \eqref{eq:ic.interpolated} obtained by interpolation are such that $\mathfrak C^0(\underline{q}_h)=0$ for all $\underline{q}_h\in\LaXgrad{h}$. By consistency of the discrete gradient $\LauGh$ and of the potential reconstructions $\LaPcurlh,\LaPgradh$, these quantities should however be small when $h$ is small and the discrete field $\underline{q}_h$ is the interpolate of a smooth field $q$ (which are the discrete fields of interest to analyse the approximation properties of the scheme). Proposition \ref{prop:constraint.preservation} then ensures that $\mathfrak C^n(\underline{q}_h)$ remain uniformly small as time evolves. 

In some circumstances however, such as to establish the energy dissipation result (see Proposition \ref{prop:energy.lm.ym} below), it can be important to ensure that $\mathfrak C^n\equiv 0$ for all $n$; by Proposition \ref{prop:constraint.preservation}, it suffices to choose initial conditions such that this relation holds for $n=0$. This can be done, for example, by solving the following linear problem, which consist in adding the constraint to \eqref{eq:ic.interpolated} and the corresponding Lagrange multiplier: set $\underline{\ymA}_h^0=\LaIcurl{h}\ymA_0$ and find $(\underline{\ymE}_h^0,\underline{\lambda}_h^0)\in(\LaXcurl{h})\times(\LaXgrad{h})$ such that
\begin{subequations}\label{eq:ic.constrained}
  \begin{alignat}{2}
    (\underline{\ymE}_h^0,\uvec{v}_h)_{\CURL,\La,h}+(\LauGh{\underline{\lambda}_h^0},\uvec{v}_h)_{\CURL,\La,h}+{}&\int_U\langle[\LaPcurlh\underline{\ymA}_h^0,\LaPgradh\underline{\lambda}_h^0],\LaPcurlh\uvec{v}_h\rangle\nonumber\\
&=(\LaIcurl{h}\ymE_0,\uvec{v}_h)_{\CURL,\La,h},\qquad\forall\uvec{v}_h\in\LaXcurl{h},\label{eq:ic.constrained.1}\\
(\underline{\ymE}_h^0,\LauGh\underline{q}_h)_{\CURL,\La,h}+\int_U\langle\LaPcurlh\underline{\ymE}_h^0,{}&[\LaPcurlh\underline{\ymA}_h^0,\LaPgradh\underline{q}_h]\rangle=0,\qquad\forall\underline{q}_h\in\LaXgrad{h}.\label{eq:ic.constrained.2}
  \end{alignat}
\end{subequations}
If $(\ymA_0,\ymE_0)$ are smooth and satisfy \eqref{eq:ym.const.E}, then solutions to \eqref{eq:ic.constrained} are consistent discretisations of these continuous initial conditions.

Making $\theta=\frac12$ in \eqref{eq:ym.lm.scheme} yields a Crank--Nicolson scheme that is consistent of order 2 in time. However, that choice does not seem to obviously lead to a stability result as in Proposition \ref{prop:energy.conservation}. In fact, we have only been able to establish an energy estimate on the scheme when $\theta=1$ and the discrete constraint is exactly satisfied.

\begin{proposition}[Energy dissipation for the constrained scheme]\label{prop:energy.lm.ym}
If $\theta=1$ in \eqref{eq:ym.lm.scheme} (implicit Euler time-stepping) and the initial conditions $(\underline{\ymA}_h^0,\underline{\ymE}_h^0)$ are such that $\mathfrak C^0\equiv 0$, then we have the decay of energy in the sense that, for all $n$,
\begin{equation}\label{eq:lm.ym.energy}
\frac{1}{2}\norm[\CURL,\La,h]{\underline{\ymE}_h^{n+1}}^2+\frac{1}{2}\norm[\DIV,\La,h]{\underline{\ymB}_h^{n+1}}^2\leq\frac{1}{2}\norm[\CURL,\La,h]{\underline{\ymE}_h^n}^2+\frac{1}{2}\norm[\DIV,\La,h]{\underline{\ymB}_h^n}^2.
\end{equation}
\end{proposition}
\begin{proof}
Substitute $\theta=1$ and $\uvec{v}_h=\underline{\ymE}_h^{n+1}$ into equation \eqref{eq:ym.lm.scheme.2}.
The Lagrange multiplier term appearing then in \eqref{eq:ym.lm.scheme} is $\mathfrak C^{n+1}(\underline{\lambda}^{n+1}_h)$, which vanishes by Proposition \ref{prop:constraint.preservation} and the choice of initial conditions that ensure $\mathfrak C^0(\underline{\lambda}^{n+1}_h)=0$. The equation \eqref{eq:ym.lm.scheme.2} with $\uvec{v}_h=\underline{\ymE}_h^{n+1}$ then reduces to \eqref{eq:unconst.2}, and the proof can be concluded as in Proposition \ref{prop:energy.conservation} with $\theta=1$.
\end{proof}

\section{Numerical tests}\label{sec:tests}

We present some numerical results of the constrained scheme \eqref{eq:ym.lm.scheme} using an implicit Euler time stepping. The DDR and LADDR complexes and the scheme were implemented using the C++ \texttt{HArDCore3D} library (see \url{https://github.com/jdroniou/HArDCore}), relying on the \texttt{Eigen3} library (see \url{http://eigen.tuxfamily.org}) for linear algebra operations. The linear systems appearing in the Newton algorithm (see below) were solved using either \texttt{Intel MKL PARADISO} (see \url{https://software.intel.com/en-us/mkl}) or \texttt{LeastSquaresConjugateGradient}, the embedded Eigen least square solver. We also used the \texttt{Spectra} library (see \url{https://spectralib.org/}) for finding the eigenvalues of the system matrices.

We consider the time domain $[0,1]$, and the space domain $U=(0,1)^3$, that we discretise using three families of meshes: Voronoi polytopal meshes, tetrahedral meshes, and cubic meshes (see Figure \ref{fig:meshes}). All the tests are run with natural boundary conditions as in \eqref{eq:bcs} -- non-homogeneous ones when we test the convergence rate to a known exact solution (Section \ref{sec:tests.convergence}), homogeneous ones when we test the constraint preservation (Section \ref{sec:tests.constraint.preservation}). We choose a uniform time step $\deltat$ such that, for each mesh of size $h$, $[0,1]$ is subdivided into $\max\{10,\left \lceil{5/h}\right \rceil\}$ time steps. The Lie algebra is taken to be $\La=\mathfrak{su}(2)$, represented as the real span of the basis $e_I=-\frac{i}{2}\sigma_I$ for $I\in\{1,2,3\}$ over $\Real$, where $\sigma_I$ are the Pauli matrices
\begin{equation*}
  \sigma_1=
  \begin{pmatrix}
    0&1\\
    1&0
  \end{pmatrix},\,
  \sigma_2=
  \begin{pmatrix}
    0&-i\\
    i&0
  \end{pmatrix},\,
  \sigma_3=
  \begin{pmatrix}
    1&0\\
    0&-1
  \end{pmatrix}.
\end{equation*}

\begin{figure}\centering
  \begin{minipage}{0.275\textwidth}
    \includegraphics[width=0.90\textwidth]{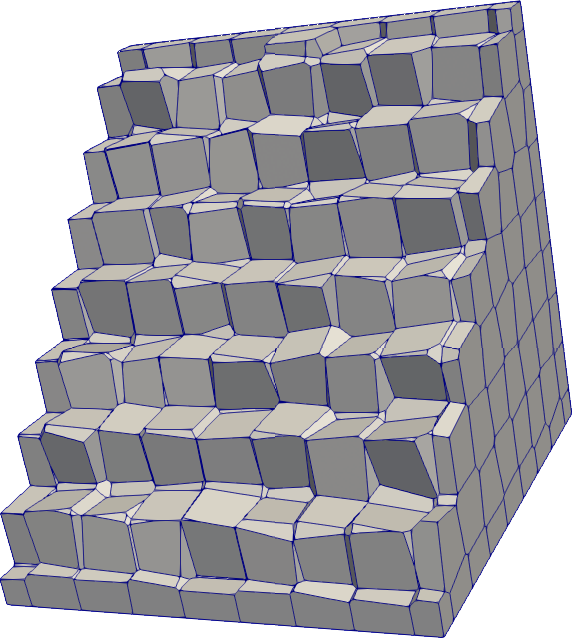}
    \subcaption{``Voro-small-0'' mesh}
  \end{minipage}
  \hspace{0.25cm}
  \begin{minipage}{0.275\textwidth}
    \includegraphics[width=0.90\textwidth]{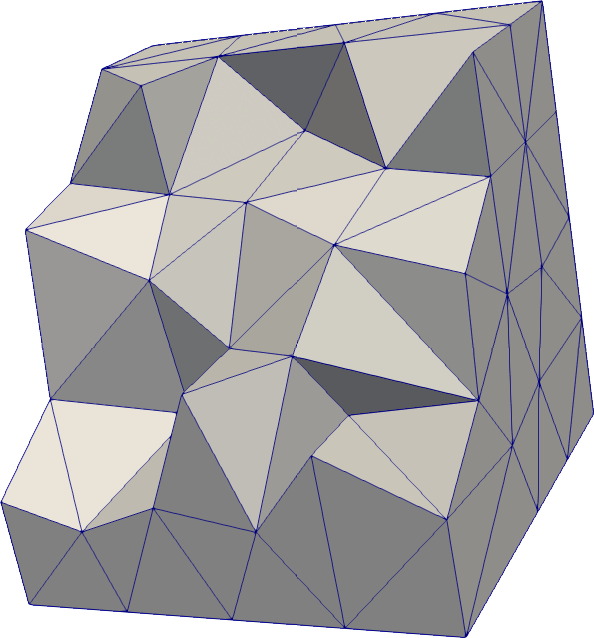}
    \subcaption{``Tetgen-Cube-0'' mesh}
  \end{minipage}
  \hspace{0.25cm}
  \begin{minipage}{0.275\textwidth}
    \includegraphics[width=0.90\textwidth]{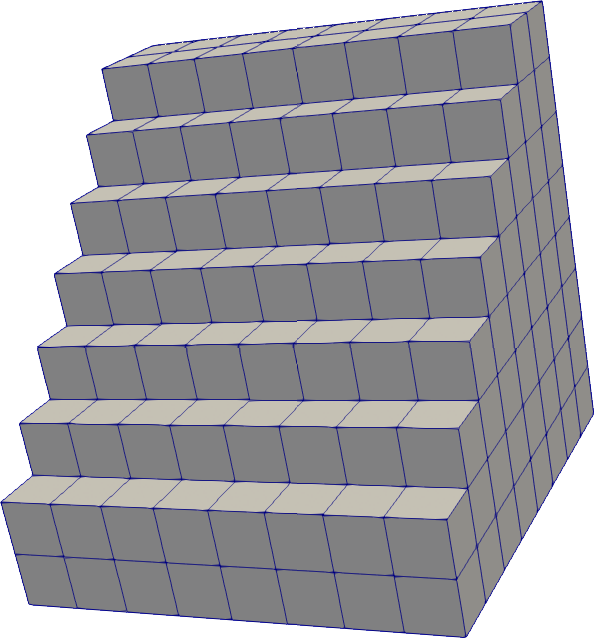}
    \subcaption{``Cubic-Cells'' mesh}
  \end{minipage}
  \caption{Meshes used in the numerical tests.}
  \label{fig:meshes}
\end{figure}

At each time, the non-linear system \eqref{eq:ym.lm.scheme} is solved using Newton iterations: writing this system as $F(\bvec{Z})=\bvec{b}$ (with $\bvec{Z}$ the vector of discrete potential and electric field, and $\bvec{b}$ encoding the unknowns at the previous time step as well as the forcing terms when relevant), we compute a sequence of solutions $(\bvec{Z}_n)_n$ such that $DF_{|\bvec{Z}_n}(\bvec{Z}_{n+1}-\bvec{Z}_n)=-F(\bvec{Z}_n)+\bvec{b}$, and we use the stopping criterion
\begin{equation*}
  \frac{\norm[l^2]{F(\bvec{Z}_{n+1})-\bvec{b}}}{\norm[l^2]{\bvec{b}}}\leq\epsilon
\end{equation*}
for $\epsilon=10^{-6}$. We noticed in our tests that, if $\deltat$ is too large compared to $h$, the Newton iterations may not converge. The time step we indicate above was found, by trial and error, as a good compromise to ensure convergence of the non-linear iterations without increasing too much the number of time steps and the resulting computational cost (note that we also take $\delta t=\mathcal O(h)$ to ensure that the truncation errors in space and time decay at the same rate). A deeper analysis of these questions, including the option of using relaxation in Newton, is the purpose of a future work; we note however that the limitation on the time step is much less severe than for, e.g., schemes for Navier--Stokes equations, and that the Newton algorithm converges much faster for the Yang--Mills equations (probably due to the nature of the non-linearity in this model, which does not include any derivatives).

We raised in Remark \ref{rem:lagrange} the question of the uniqueness of the Lagrange multiplier $\lambda$ in the constrained scheme. We numerically checked the invertibility of the matrices $M=DF(\bvec{Z})$ involved in the Newton iterations, by computing the minimum eigenvalue of $M^tM$.
In many cases, this eigenvalue was close to machine precision, indicating that the linear system is probably singular, and there are infinitely many solutions. This singularity can be problematic for certain solvers (in particular, Pardiso is not ensured to solve a singular system); in practice, given the energy estimate \eqref{prop:energy.lm.ym} on the potential and electric field, we believe that the singularity occurs mostly on the components related to $\lambda$, indicating that this Lagrange multiplier may not be unique (imposing a zero integral on this Lagrange multiplier led to the invertibility of the augmented linear system, but this constraint cannot be imposed in the scheme \eqref{eq:ym.lm.scheme} without losing the preservation of constraint -- this can be theoretically verified, and we also checked it numerically). Exploring variants of the scheme or the non-linear iterations which ensure the well-posedness of each linear system is a matter that remains to be explored (and does not seem to have been considered in the previous work using the similar constrained equation \cite{Christiansen.Winther:06}). For all the results presented here, we found that Pardiso (for Voronoi meshes) or the Eigen least-square solver (for tetrahedral and cubic meshes) managed to find an acceptable solution (residuals of the linear systems below $10^{-12}$ for most cases, with the worst one at $10^{-9}$ for the fourth mesh in the Voronoi family).

We use for the rest of this section, the gauge potential $\ymA(t)$: 
\begin{equation}\label{eq:gauge.A}
\begin{aligned}
  \ymA(t)={}&
  \begin{bmatrix}
    -0.5\cos(t)\sin(\pi x)\cos(\pi y)\cos(\pi z)\\ 
    \cos(t)\cos(\pi x)\sin(\pi y)\cos(\pi z)\\ 
    -0.5\cos(t)\cos(\pi x)\cos(\pi y)\sin(\pi z)
  \end{bmatrix}
  \otimes e_1 +
  \begin{bmatrix}
    -0.5\sin(t)\sin(\pi x)\cos(\pi y)\cos(\pi z)\\ 
    \sin(t)\cos(\pi x)\sin(\pi y)\cos(\pi z)\\ 
    -0.5\sin(t)\cos(\pi x)\cos(\pi y)\sin(\pi z)
  \end{bmatrix}
  \otimes e_2 \\
  &+
  \begin{bmatrix}
    -0.5\sin(t)\sin^2(\pi y)\\ 
    \cos(t)\cos^2(\pi z)\\ 
    -0.5\sin(t)\cos^2(\pi x)
  \end{bmatrix}\otimes e_3,
\end{aligned}
\end{equation}
from which all remaining functions and (non-homogeneous) boundary conditions can be derived.

\begin{table}\centering
\begin{tabular}{c|ccccc||ccccc|}
 & \multicolumn{5}{c||}{Voronoi mesh} & \multicolumn{5}{c|}{Tetrahedral mesh}\\
 & 1 & 2 & 3 & 4 & 5 & 1 & 2 & 3 & 4 & 5 \\
 \hline
$h$ & 0.83 & 0.45 & 0.31 & 0.22 & 0.18 & 0.56 & 0.50 & 0.39 & 0.31 & 0.26 \\
$\deltat$ & 0.1 & 0.083 & 0.059 & 0.043 & 0.034 & 0.1 & 0.091 & 0.077 & 0.063 & 0.05 \\
$N_{\rm avg}\;(\epsilon=10^{-6})$ & 2 & 2 & 2 & 2.6 & 1.4 & 2 & 2 & 2 & 1.9 & 1.6 \\
$N_{\rm tot}\;(\epsilon=10^{-6})$ & 20 & 24 & 34 & 60 & 42 & 20 & 22 & 26 & 31 & 32 \\
$N_{\rm avg}\;(\epsilon=10^{-10})$ & 2.3 & 2.3 & 2.1 & 3.3 & 2 & 2 & 2 & 2 & 2 & 2 \\
$N_{\rm tot}\;(\epsilon=10^{-10})$ & 23 & 27 & 36 & 76 & 59 & 20 & 22 & 26 & 32 & 40 \\ 
\end{tabular}
\caption{Numerical parameters and Newton iterations (averaged over the time steps $N_{\rm avg}$, and total $N_{\rm tot}$ over all time steps) the for the tests in Section \ref{sec:tests.convergence}, for values $\epsilon=10^{-6},10^{-10}$ of the stopping criterion.}
\label{tab:num.para.test}
\end{table}

\subsection{Convergence tests}\label{sec:tests.convergence}

To test the convergence, we modify the scheme \eqref{eq:ym.lm.scheme} by adding the forcing term $(\LaIcurl{h}\bvec{F}(t), \uvec{v})_{\CURL,\La,h}$ to \eqref{eq:ym.lm.scheme.2} where $\bvec{F} = \partial_t\ymE-\CURL\ymB-\ebkt{\ymA}{\ymB}$, and $\frac1\deltat(\LaIcurl{h}\ymE(t+\deltat)-\LaIcurl{h}\ymE(t), \uvec{v})_{\CURL,\La,h}$ to \eqref{eq:ym.lm.scheme.3} to ensure we solve for the fields derived from \eqref{eq:gauge.A}. The initial conditions are the interpolates computed via \eqref{eq:ic.interpolated}, and we compare our solution at $t=1$ with the interpolates $\LaIcurl{h}\ymA(1)$, $\LaIcurl{h}\ymE(1)$ in the $\LaXcurl{h}$ norm. The relative errors for the three families of meshes are shown in Figure \ref{fig:errors} where, from our choice of $\deltat$, we see that the scheme converges at the expected rate of 1 (actually, slightly above 1 for the error in $\ymE$ on tetrahedral meshes, but probably due to the fact that the asymptotic regime has not been reached in these tests).

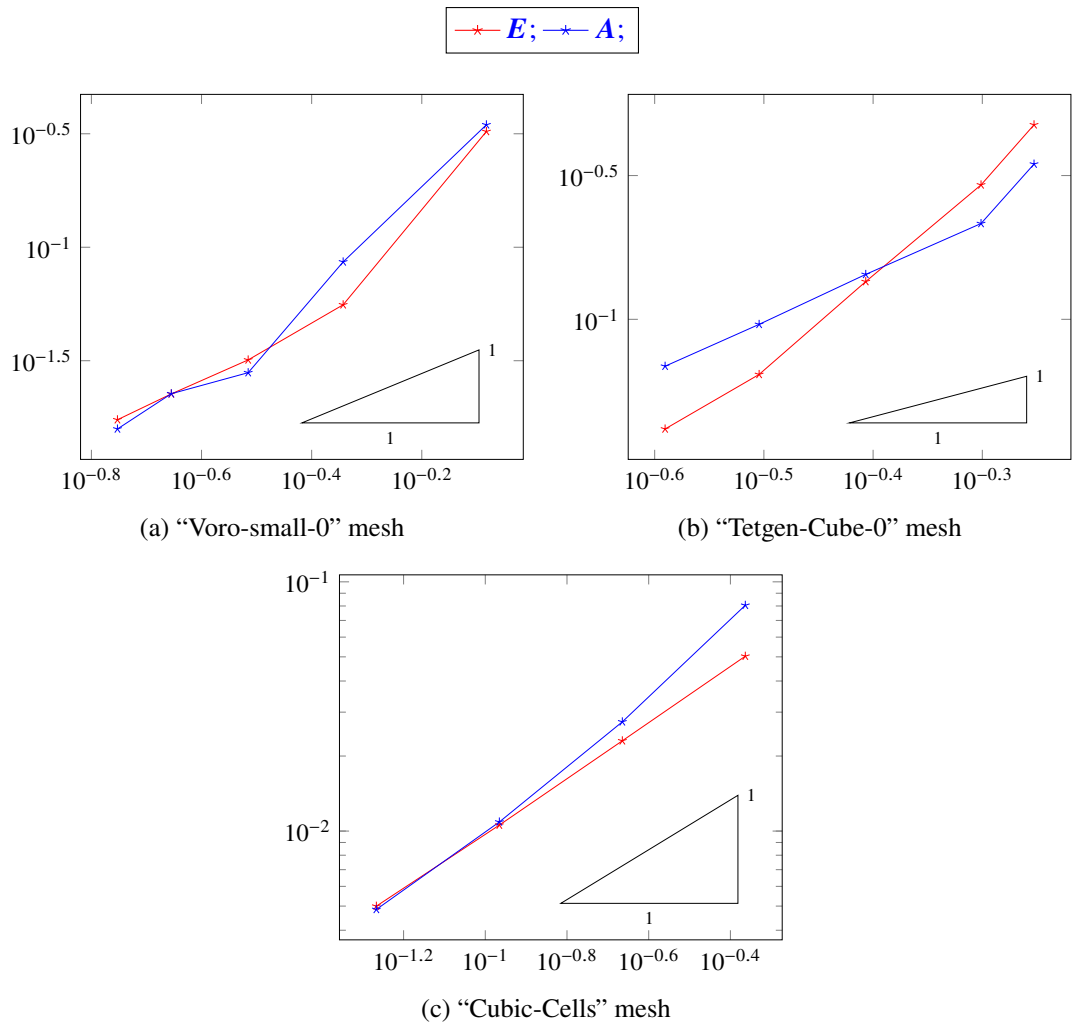
\begin{figure}\centering
  \ref{elec.pot.fields}
  \vspace{0.50cm}\\
  \begin{minipage}{0.45\textwidth}
    \begin{tikzpicture}[scale=0.85]
      \begin{loglogaxis} [legend columns=3, legend to name=elec.pot.fields]  
        \logLogSlopeTriangle{0.90}{0.4}{0.1}{1}{black};
        \addplot [mark=star, red] table[x=MeshSize,y=E_L2Elec] {outputs/errors/Voro-small-0_k0/data_rates.dat};
        \addlegendentry{$\ymE$;}
        \addplot [mark=star, blue] table[x=MeshSize,y=E_L2Pot] {outputs/errors/Voro-small-0_k0/data_rates.dat};
        \addlegendentry{$\ymA$;}
      \end{loglogaxis}            
    \end{tikzpicture}
    \subcaption{``Voro-small-0'' mesh}
  \end{minipage}
  \begin{minipage}{0.45\textwidth}
    \begin{tikzpicture}[scale=0.85] 
      \begin{loglogaxis}
        \logLogSlopeTriangle{0.90}{0.4}{0.1}{1}{black};
        \addplot [mark=star, red] table[x=MeshSize,y=E_L2Elec] {outputs/errors/Tetgen-Cube-0_k0/data_rates.dat};
        \addplot [mark=star, blue] table[x=MeshSize,y=E_L2Pot] {outputs/errors/Tetgen-Cube-0_k0/data_rates.dat};           
        \end{loglogaxis} 
      \end{tikzpicture}
    \subcaption{``Tetgen-Cube-0'' mesh}
  \end{minipage}\\[0.5em]
  \begin{minipage}{0.45\textwidth}
    \begin{tikzpicture}[scale=0.85] 
      \begin{loglogaxis}
        \logLogSlopeTriangle{0.90}{0.4}{0.1}{1}{black};
        \addplot [mark=star, red] table[x=MeshSize,y=E_L2Elec] {outputs/errors/Cubic-Cells_k0/data_rates.dat};
        \addplot [mark=star, blue] table[x=MeshSize,y=E_L2Pot] {outputs/errors/Cubic-Cells_k0/data_rates.dat};           
        \end{loglogaxis} 
      \end{tikzpicture}
    \subcaption{``Cubic-Cells'' mesh}
  \end{minipage}\\[0.5em]
  \caption{Tests of Section \ref{sec:tests.convergence}: relative errors on Voronoi, tetrahedral and cubic meshes}
  \label{fig:errors}
\end{figure}

We record also the effect of mesh size and a smaller stopping value ($\epsilon=10^{-10}$) on the Newton iterations in these tests. The results in Table \ref{tab:num.para.test} show that a rather small number of non-linear iterations is actually required; interestingly, this number does not increase as the mesh size decreases, and does not change much either when $\epsilon$ is reduced. This demonstrates that, on this problem, the straight Newton solver (without the need for relaxation or other adjustments) is a very efficient one -- provided as explained above that the time step is not too large.

\subsection{Constraint preservation}\label{sec:tests.constraint.preservation}

We now present tests to assess the preservation of constraint property of the scheme \eqref{eq:ym.lm.scheme}. Our initial conditions are built on the gauge \eqref{eq:gauge.A}, either by simple interpolation as in \eqref{eq:ic.interpolated} or by projecting them through \eqref{eq:ic.constrained} to ensure the zero initial constraint $\mathfrak C^0\equiv 0$. We present the results in Figure \ref{fig:constraint}, in which we plot against the mesh size the maximum over $n$ of the norm (in the dual space of $\LaXgrad{h}$) of $\mathfrak C^n-\mathfrak C^0$. As expected in both cases, the norm of this difference stays close to machine precision, accounting for some accumulation of the rounding errors as $h$ decreases.

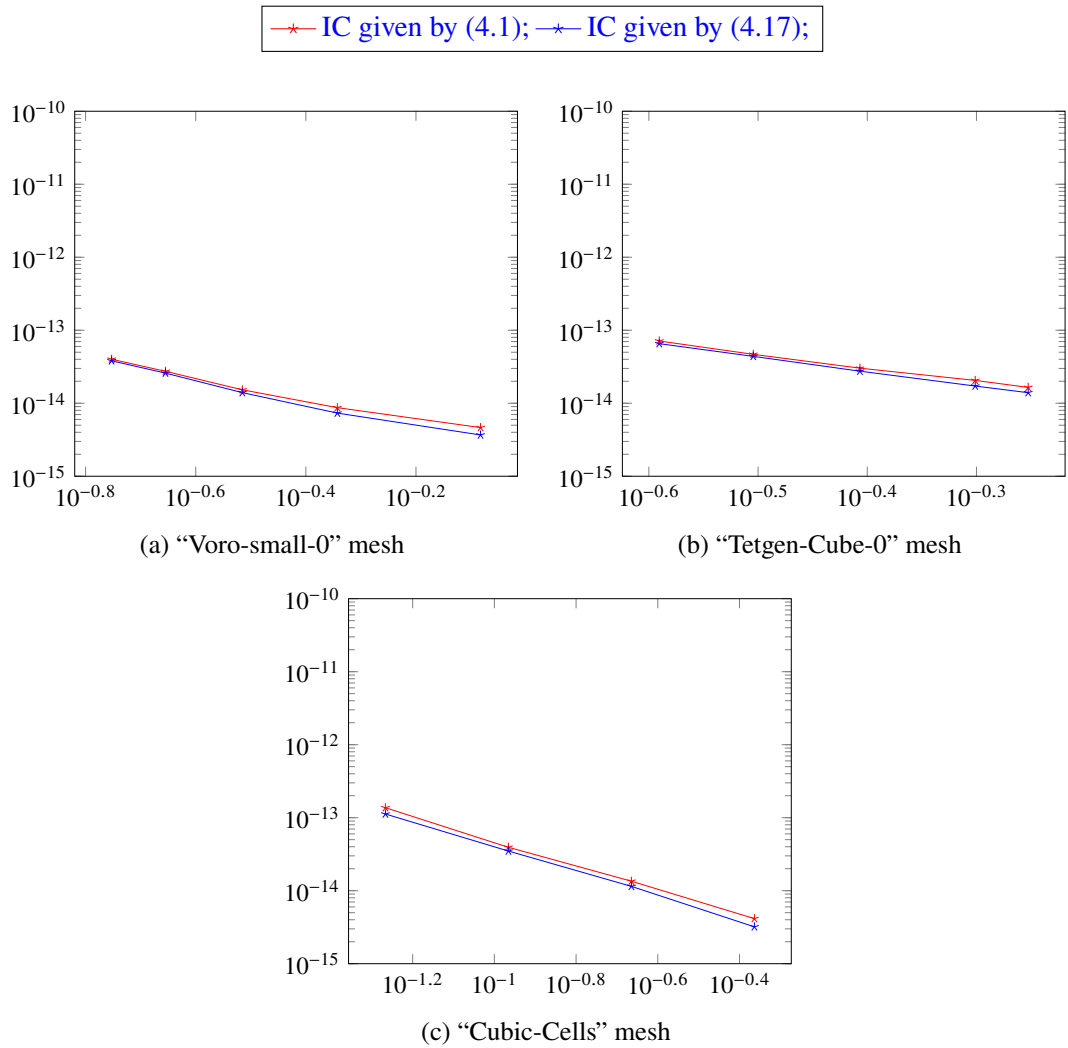
\begin{figure}\centering
  \ref{constraint.zero}
  \vspace{0.50cm}\\
  \begin{minipage}{0.45\textwidth}
    \begin{tikzpicture}[scale=0.85]
      \begin{loglogaxis} [legend columns=3, legend to name=constraint.zero, ymin=1e-15, ymax=1e-10]
        \addplot [mark=star, red] table[x=MeshSize,y=MaxConstr] {outputs/constr-preserve/Voro-small-0_k0/data_rates.dat};
        \addlegendentry{IC given by \eqref{eq:ic.interpolated};}
        \addplot [mark=star, blue] table[x=MeshSize,y=MaxConstr] {outputs/IC-constr-preserve/Voro-small-0_k0/data_rates.dat};
        \addlegendentry{IC given by \eqref{eq:ic.constrained};}
      \end{loglogaxis}            
    \end{tikzpicture}
    \subcaption{``Voro-small-0'' mesh}
  \end{minipage}
  \begin{minipage}{0.45\textwidth}
    \begin{tikzpicture}[scale=0.85] 
      \begin{loglogaxis}[ymin=1e-15, ymax=1e-10]
        \addplot [mark=star, red] table[x=MeshSize,y=MaxConstr] {outputs/constr-preserve/Tetgen-Cube-0_k0/data_rates.dat};
        \addplot [mark=star, blue] table[x=MeshSize,y=MaxConstr] {outputs/IC-constr-preserve/Tetgen-Cube-0_k0/data_rates.dat};           
        \end{loglogaxis} 
      \end{tikzpicture}
    \subcaption{``Tetgen-Cube-0'' mesh}
  \end{minipage}\\[0.5em]
  \begin{minipage}{0.45\textwidth}
    \begin{tikzpicture}[scale=0.85] 
      \begin{loglogaxis}[ymin=1e-15, ymax=1e-10]
        \addplot [mark=star, red] table[x=MeshSize,y=MaxConstr] {outputs/constr-preserve/Cubic-Cells_k0/data_rates.dat};
        \addplot [mark=star, blue] table[x=MeshSize,y=MaxConstr] {outputs/IC-constr-preserve/Cubic-Cells_k0/data_rates.dat};           
        \end{loglogaxis} 
      \end{tikzpicture}
    \subcaption{``Cubic-Cells'' mesh}
  \end{minipage}
  \caption{Tests of Section \ref{sec:tests.constraint.preservation} (preservation of constraint): maximum over $n$ of the norms of $\mathfrak C^n-\mathfrak C^0$ against mesh sizes on Voronoi, tetrahedral and cubic meshes}
  \label{fig:constraint}
\end{figure}

\section{Extension to higher order}\label{sec:DDR.high.order}

\subsection{Arbitrary order DDR and LADDR complexes}

The presentation of the DDR spaces for an arbitrary order of accuracy requires the introduction of the following polynomial sub-spaces, in which $\ell\ge 0$ is a natural number (and we recall that $\bvec{x}_\sfP$ is a point on $\sfP\in\{F,T\}$): for a face $F\in\Fh$,
\begin{alignat*}{2}
    \Goly{\ell}(F)&\coloneq\GRAD_F\Poly{\ell+1}(F),
    &\qquad
    \cGoly{\ell}(F)&\coloneq(\bvec{x}-\bvec{x}_F)^\perp\Poly{\ell-1}(F),
    \\ 
    \Roly{\ell}(F)&\coloneq\VROT_F\Poly{\ell+1}(F),
    &\qquad
    \cRoly{\ell}(F)&\coloneq(\bvec{x}-\bvec{x}_F)\Poly{\ell-1}(F),
\end{alignat*}
where $(\bvec{x}-\bvec{x}_F)^\perp$ denotes the vector $\bvec{x}-\bvec{x}_F$ rotated by $-\pi/2$ in the plane of $F$
oriented by $\normal_F$ and, for an element $T\in\Th$,
\begin{alignat*}{2}
    \Goly{\ell}(T)&\coloneq\GRAD\Poly{\ell+1}(T),
    &\qquad 
    \cGoly{\ell}(T)&\coloneq(\bvec{x}-\bvec{x}_T)\times \vPoly{\ell-1}(T),
    \\ 
    \Roly{\ell}(T)&\coloneq\CURL\vPoly{\ell+1}(T),
    &\qquad
    \cRoly{\ell}(T)&\coloneq(\bvec{x}-\bvec{x}_T)\Poly{\ell-1}(T).
\end{alignat*}
If $\ell=0$ we set $\Poly{-1}(F)=\Poly{-1}(T)=\{0\}$. For $\sfP\in\{F,T\}$ and $\cvec{X}\in\{\cvec{R},\cvec{G}\}$, $\cvec{X}^{\compl,\ell}(\sfP)$ is the Koszul complement of $\cvec{X}^\ell(\sfP)$ in $\vPoly{\ell}(\sfP)\coloneq \Poly{\ell}(\sfP)^d$, where $d$ is the dimension of $\sfP$ (in the case $\sfP=F$, elements of $\vPoly{\ell}(F)$ are considered tangential to $F$). We denote by $\Xproj{\ell}{\sfP}$ and $\Xcproj{\ell}{\sfP}$ the orthogonal projectors, respectively, on $\cvec{X}^\ell(\sfP)$ and $\cvec{X}^{\compl,\ell}(\sfP)$.

For a given polynomial degree $k\ge 0$ (representing the degree of accuracy of the complex), the DDR spaces are
\begin{equation*}
  \Xgrad[k]{h}\coloneq\Big\{
  \begin{aligned}[t]
  \underline{q}_h&=\big((q_T)_{T\in\Th},(q_F)_{F\in\Fh}, q_{\Eh}\big)\st
    \text{$q_T\in \Poly{k-1}(T)$ for all $T\in\Th$,}\\
    &\text{$q_F\in\Poly{k-1}(F)$ for all $F\in\Fh$,}\\
    &\text{and $q_{\Eh}\in C^0(\Eh)$ is such that $(q_{\Eh})_{|E}\in\Poly{k+1}(E)\quad\forall E\in\Eh$}
    \Big\},
  \end{aligned}
\end{equation*}
\begin{equation*}
  \Xcurl[k]{h}\coloneq\Big\{
  \begin{aligned}[t]
    \uvec{v}_h
    &=\big(
    (\bvec{v}_{\cvec{R},T},\bvec{v}_{\cvec{R},T}^\compl)_{T\in\Th},(\bvec{v}_{\cvec{R},F},\bvec{v}_{\cvec{R},F}^\compl)_{F\in\Fh}, (v_E)_{E\in\Eh}
    \big)\st
    \\
    &\qquad\text{$\bvec{v}_{\cvec{R},T}\in\Roly{k-1}(T)$ and $\bvec{v}_{\cvec{R},T}^\compl\in\cRoly{k}(T)$ for all $T\in\Th$,}
    \\
    &\qquad\text{$\bvec{v}_{\cvec{R},F}\in\Roly{k-1}(F)$ and $\bvec{v}_{\cvec{R},F}^\compl\in\cRoly{k}(F)$ for all $F\in\Fh$,}
    \\
    &\qquad\text{and $v_E\in\Poly{k}(E)$ for all $E\in\Eh$}\Big\},
  \end{aligned}
\end{equation*}
\begin{equation*}
  \Xdiv[k]{h}\coloneq\Big\{
  \begin{aligned}[t]
    \uvec{w}_h
    &=\big((\bvec{w}_{\cvec{G},T},\bvec{w}_{\cvec{G},T}^\compl)_{T\in\Th}, (w_F)_{F\in\Fh}\big)\st
    \\
    &\qquad\text{$\bvec{w}_{\cvec{G},T}\in\Goly{k-1}(T)$ and $\bvec{w}_{\cvec{G},T}^\compl\in\cGoly{k}(T)$ for all $T\in\Th$,}
    \\
    &\qquad\text{and $w_F\in\Poly{k}(F)$ for all $F\in\Fh$}
    \Big\},
  \end{aligned}
\end{equation*}
and
\[
\Poly{k}(\Th)\coloneq\left\{
q_h\in L^2(U)\st\text{$(q_h)_{|T}\in\Poly{k}(T)$ for all $T\in\Th$}
\right\}.
\]
The interpolators on $\Xcurl[k]{h}$, $\Xdiv[k]{h}$ and $\Poly{k}(\Th)$ are defined as the $L^2$-projections on the corresponding polynomial subspaces; the face and element components of the interpolator on $\Xgrad[k]{h}$ are also based on $L^2$-projections, while the skeletal (edge) component is built from pointwise values at the vertices and projections on $\Poly{k-1}(E)$ for each $E\in\Eh$.
Discrete gradient, curl and divergence operators can then be designed between these spaces such that the following sequence forms a complex with the same cohomology as the continuous de Rham complex (see \cite{Di-Pietro.Droniou:21*1,Di-Pietro.Droniou.ea:22} for details of the construction and of the cohomology analysis):
\begin{equation*}
  \begin{tikzcd}
    \Real\arrow{r}{\Igrad[k]{h}} & \Xgrad[k]{h}\arrow{r}{\uGh[k]} & \Xcurl[k]{h}\arrow{r}{\uCh[k]} & \Xdiv[k]{h}\arrow{r}{\Dh[k]} & \Poly{k}(\Th)\arrow{r}{0} & \{0\}.
  \end{tikzcd}
\end{equation*}
As for the low-order DDR sequence, potential reconstructions (of degree $k+1$ on $\Xgrad[k]{h}$, and degree $k$ on $\Xcurl[k]{h}$ and $\Xdiv[k]{h}$) and inner products can also be designed on the spaces of this complex.

\medskip

The construction of an LADDR complex from this $k$-th order DDR complex is done by tensorisation of the components and operators, exactly as for the lowest-order LADDR complex of Section \ref{sec:LADDR}. So, for example,
\begin{equation*}
  \LaXcurl[k]{h}\coloneq\Big\{
  \begin{aligned}[t]
    \uvec{v}_h
    &=\big(
    (\bvec{v}_{\cvec{R},T},\bvec{v}_{\cvec{R},T}^\compl)_{T\in\Th},(\bvec{v}_{\cvec{R},F},\bvec{v}_{\cvec{R},F}^\compl)_{F\in\Fh}, (v_E)_{E\in\Eh}
    \big)\st
    \\
    &\qquad\text{$\bvec{v}_{\cvec{R},T}\in\Roly{k-1}(T)\otimes\La$ and $\bvec{v}_{\cvec{R},T}^\compl\in\cRoly{k}(T)\otimes\La$ for all $T\in\Th$,}
    \\
    &\qquad\text{$\bvec{v}_{\cvec{R},F}\in\Roly{k-1}(F)\otimes\La$ and $\bvec{v}_{\cvec{R},F}^\compl\in\cRoly{k}(F)\otimes\La$ for all $F\in\Fh$,}
    \\
    &\qquad\text{and $v_E\in\Poly{k}(E)\otimes\La$ for all $E\in\Eh$}\Big\}
  \end{aligned}
\end{equation*}
and the LADDR complex reads
\begin{equation*}
  \begin{tikzcd}
    \Real\otimes\La\arrow{r}{\LaIgrad[\La,k]{h}} & \LaXgrad[k]{h}\arrow{r}{\LauGh[\La,k]} & \LaXcurl[k]{h}\arrow{r}{\LauCh[\La,k]} & \LaXdiv[k]{h}\arrow{r}{\LaDh[\La,k]} & \Poly{k}(\Th)\otimes\La\arrow{r}{0} & \{0\},
  \end{tikzcd}
\end{equation*}
where the operators are defined as before by tensorisation of the scalar operators together with the identity on the Lie algebra component. Hence, for all $\uvec{v}_h=\uvec{v}_h^I\otimes e_I\in \LaXcurl[k]{h}$ with $\uvec{v}_h^I\in\Xcurl[k]{h}$, we set $\LauCh[\La,k]\uvec{v}_h= (\uCh[k]\uvec{v}_h^I)\otimes e_I$. Lie algebra-valued potential reconstructions are designed the same way, and we obtain inner products $(\cdot,\cdot)_{\bullet,\La,k,h}$ by applying the construction \eqref{eq:def.la.inner} using the inner product on the high-order scalar spaces $\Xbullet[k]{h}$.

\subsection{Scheme}

Schemes using the arbitrary-order LADDR complex defined above are obtained following the same principles as for the low order version: the continuous spaces, operators and $L^2$-inner products are replaced by the discrete (high-order) spaces, operators and inner products. For the specific case of the Yang--Mills equations, we also need to reconstruct a version of $\Hstar[\cdot,\cdot]$ in $\LaXdiv[k]{h}$. This can be done by defining $\ebkttrk{\cdot}{\cdot}:(\LaXcurl[k]{h})\times(\LaXcurl[k]{h})\to\LaXdiv[k]{h}$ such that, for all $\uvec{v}_h,\uvec{w}_h\in\LaXcurl[k]{h}$, all $F\in\Fh$ and $T\in\Th$, \begin{subequations}\label{eq:def.ebkttrk}
\begin{alignat}{2}
\left(\ebkttrk{\uvec{v}_h}{\uvec{w}_h}\right)_F={}&\lproj{k}{F}\left(\ebkt{\LatrFt[\La,k]\uvec{v}_F}{\LatrFt[\La,k]\uvec{w}_F}\cdot\normal_F\right)\in\Poly{k}(F)\otimes\La,
\label{eq:def.ebkttrk.defF}\\
\left(\ebkttrk{\uvec{v}_h}{\uvec{w}_h}\right)_{\cvec{G},T}={}&\Gproj{k-1}{T}\left(\ebkt{\LaPcurl[\La,k]\uvec{v}_T}{\LaPcurl[\La,k]\uvec{w}_T}\right)\in\Goly{k-1}(T)\otimes \La,
\label{eq:def.ebkttrk.def.GolyT}\\
\left(\ebkttrk{\uvec{v}_h}{\uvec{w}_h}\right)_{\cvec{G},T}^\compl={}&\Gcproj{k}{T}\left(\ebkt{\LaPcurl[\La,k]\uvec{v}_T}{\LaPcurl[\La,k]\uvec{w}_T}\right)\in\cGoly{k}(T)\otimes\La.
\label{eq:def.ebkttrk.def.cGolyT}
\end{alignat}
\end{subequations}
Here, the projectors $\lproj{k}{F}$, $\Gproj{k-1}{T}$ and $\Gcproj{k}{T}$ act, respectively, on the $\Poly{2k}(F)$ and $\vPoly{2k}(T)$ components of their arguments.

The constrained scheme, high-order equivalent of \eqref{eq:ym.lm.scheme}, then reads: Find families $(\underline{\ymA}_h^n)_n$, $(\underline{\ymE}_h^n)_n$, $(\underline{\lambda}_h^n)_n$ where $(\underline{\ymA}_h^n,\underline{\ymE}_h^n,\underline{\lambda}_h^n)\in(\LaXcurl[k]{h})\times(\LaXcurl[k]{h})\times(\LaXgrad[k]{h})$, such that, for all $n$, setting $\underline{\ymB}_h^{n+\theta}=\LauCh[\La,k]\underline{\ymA}_h^{n+\theta}+\frac12 \ebkttrk{\underline{\ymA}_h^{n+\theta}}{\underline{\ymA}_h^{n+\theta}}\in\LaXdiv[k]{h}$,
\begin{subequations}\label{eq:ym.lm.scheme.k}
\begin{alignat}{4}
\delta_t^{n+1}\underline{\ymA}_h&=-\underline{\ymE}_h^{n+\theta},\label{eq:ym.lm.scheme.k.1}\\ 
(\delta_t^{n+1}\underline{\ymE}_h,\uvec{v}_h)_{\CURL,\La,k,h}+{}&(\LauGh[\La,k]{\underline{\lambda}_h^{n+1}},\uvec{v}_h)_{\CURL,\La,k,h}+\int_U\langle[\LaPcurlh[\La,k]\underline{\ymA}_h^{n+\theta},\LaPgradh[\La,k+1]\underline{\lambda}_h^{n+1}],\LaPcurlh[\La,k]\uvec{v}_h\rangle\nonumber\\
=(\underline{\ymB}_h^{n+\theta},{}&\LauCh[\La,k]\uvec{v}_h+\ebkttrk{\underline{\ymA}_h^{n+\frac{1}{2}}}{\uvec{v}_h})_{\DIV,\La,k,h},\qquad\forall\uvec{v}_h\in\LaXcurl[k]{h},\label{eq:ym.lm.scheme.k.2}\\
(\delta_t^{n+1}\underline{\ymE}_h,\LauGh[\La,k]\underline{q}_h)_{\CURL,\La,k,h}{}&+\int_U\langle\LaPcurlh[\La,k](\delta_t^{n+1}\underline{\ymE}_h),[\LaPcurlh[\La,k]\underline{\ymA}_h^{n+1-\theta},\LaPgradh[\La,k+1]\underline{q}_h]\rangle\nonumber\\
&=0,\qquad\forall\underline{q}_h\in\LaXgrad[k]{h}.\label{eq:ym.lm.scheme.k.3}
\end{alignat}
\end{subequations}

The proofs of constraint preservation (Proposition \ref{prop:constraint.preservation}) and energy estimates (Proposition \ref{prop:energy.lm.ym}) do not rely on any particularity of the lowest-order LADDR complex, and these results are therefore still valid for \eqref{eq:ym.lm.scheme.k}, with obvious substitutions: the norms in \eqref{eq:lm.ym.energy} should be those of $\LaXcurl[k]{h}$ and $\LaXdiv[k]{h}$, and the discrete constraint \eqref{discrete.const} is defined using $\LauGh[\La,k]$, $\LaPcurlh[\La,k]$, $\LaPgrad[\La,k+1]$ and the inner product in $\LaXcurl[k]{h}$.
For the energy estimate, initial conditions $(\underline{\ymA}_h^0,\underline{\ymE}_h^0)\in(\LaXcurl[k]{h})\times(\LaXcurl[k]{h})$ with a vanishing discrete constraint can be obtained by solving the high-order equivalent of \eqref{eq:ic.constrained}.


\section{Conclusion}\label{sec:conclusion}

We designed a scheme for the Yang--Mills equations, in flat 3+1 formalism, based on the discrete de Rham complex. This discretisation can be applied on generic polyhedral meshes, which allows, for example, for seamless local mesh refinement; it also has an arbitrary-order variant. A novelty of the approach is the treatment at the discrete level of the various Lie brackets appearing in the weak formulation of the equations; in particular, we show how to obtain suitable discrete Lie brackets that enable a proper definition of the magnetic field, and proofs of the discrete preservation of constraint, as well as energy estimates for the implicit time stepping. The resulting system of equations is fully non-linear, due to the co-location in time of the potential and electric field, but can be efficiently solved using a straightforward Newton algorithm. We provided results of tests on a 3D domain for the lowest-order version of the scheme, showing the expected rate of convergence and preservation of the discrete constraint.

It should be noted that our approach, although based on the DDR as underlying (real-valued) polytopal method, can be easily adapted to other polytopal methods such as the Virtual Element Method.
Future work will tackle the question of convergence analysis of this scheme, and of numerically testing its higher-order version.
We also note that \cite{Christiansen:08-09} extends the approach of \cite{Christiansen.Winther:06} for conforming numerical approximations of more general wave equations than the Yang--Mills equations. We believe that, in a similar way, the DDR approach could be extended to such models.


\printbibliography

\end{document}